\documentclass[12pt]{iopart}

\usepackage{iopams}
\usepackage{graphicx,psfrag,epsf}
\usepackage{amsthm,amssymb}
\usepackage{algorithm}
\usepackage{algorithmicx}
\usepackage{algpseudocode}
\usepackage{subfigure}
\usepackage{setstack}
\usepackage{array}
\usepackage[a4paper,left=2.1cm,right=2.1cm,top=2.3cm,bottom=2.3cm]{geometry}%
\usepackage{xcolor}
\usepackage{multirow}
\usepackage{url}

\makeatletter

\newcommand{\Rmnum}[1]{\expandafter\@slowromancap\romannumeral #1@}

\makeatother

\newtheorem{proposition}{Proposition}
\newtheorem{lemma}{Lemma}
\newtheorem{theorem}{Theorem}
\newtheorem{remark}{Remark}
\newtheorem{problem}{Problem}
\newtheorem{assumption}{Assumption}

\begin{document}

\title[Accelerated regularization methods]{A new class of accelerated regularization methods, with application to bioluminescence tomography}


\author{Rongfang Gong$^{1}$, Bernd Hofmann$^{2}$ and Ye Zhang$^{3,4}$\footnote{Corresponding Author}}

\address{$^1$Department of Mathematics, Nanjing University of Aeronautics and Astronautics, 211106 Nanjing, China}
\address{$^2$Faculty of Mathematics, Chemnitz University of Technology, 09107 Chemnitz, Germany}
\address{$^3$Shenzhen MSU-BIT University, 518172 Shenzhen, China}
\address{$^4$School of Mathematics and Statistics, Beijing Institute of Technology, 100081 Beijing, China}
\ead{grf\_math@nuaa.edu.cn, hofmannb@mathematik.tu-chemnitz.de, ye.zhang@oru.se}
\vspace{10pt}

\begin{abstract}
In this paper we propose a new class of iterative regularization methods for solving ill-posed linear operator equations. The prototype of these iterative regularization methods is in the form of second order evolution equation with a linear vanishing damping term, which can be viewed not only as an extension of the asymptotical regularization, but also as a continuous analog of the Nesterov's acceleration scheme. New iterative regularization methods are derived from this continuous model in combination with damped symplectic numerical schemes. The regularization property as well as convergence rates and acceleration effects under the H\"older-type source conditions of both continuous and discretized methods are proven.

The second part of this paper is concerned with the application of the newly developed accelerated iterative regularization methods to the diffusion-based bioluminescence tomography, which is modeled as an inverse source problem in elliptic partial differential equations with both Dirichlet and Neumann boundary data. A relaxed mathematical formulation is proposed so that the discrepancy principle can be applied to the iterative scheme without the usage of Sobolev embedding constants. Several numerical examples, as well as a comparison with the state-of-the-art methods, are given to show the accuracy and the acceleration effect of the new methods.
\end{abstract}


\section{Introduction}
\label{Introduction}

In the first part of this paper we consider linear operator equations
\begin{equation}\label{OperatorEq}
K f = y,
\end{equation}
where $K$ is a compact linear operator acting between two infinite dimensional Hilbert spaces $Q$ and $Y$ such that the range $\mathcal{R}(K)$ of $K$ is an infinite dimensional subspace of $Y$. Then
the range $\mathcal{R}(K)$ is a non-closed subset of $Y$. For simplicity, we denote by $\langle \cdot, \cdot \rangle$ and $\|\cdot\|$ in the sequel the inner products and norms for both Hilbert spaces $Q$ and $Y$.
The non-closedness of the forward operator $K$ is typical for operator equations (\ref{OperatorEq}) which are models for {\it linear inverse problems}. More precisely, due to the compactness of $K$, the operator equation (\ref{OperatorEq}) is ill-posed of type II in the sense of Nashed (cf.~\cite{Nashed}). As a consequence of this ill-posedness, a regularization method must be employed in order to obtain reasonable
 and stable approximate solutions to (\ref{OperatorEq}) if the measurement data contains noise. In this context, we consider {\it iterative regularization methods} and assume to know noisy data $y^\delta \in Y$ instead
 of the exact right-hand side $y \in \mathcal{R}(K)$, obeying
 the deterministic noise model $\|y^\delta - y\|\leq \delta$ with a priori known noise level $\delta>0$. The focus of our paper will be on studying a specific application in bioluminescence tomography,
 and we refer to Section~4 for details.

The dominant iterative regularization method for solving (\ref{OperatorEq}) should be the Landweber method, given by
\begin{equation}\label{Landweber}
f^\delta_{k+1} = f^\delta_{k} + \Delta t K^* ( y^\delta- K f^\delta_{k} ), \quad \Delta t\in(0, 2/\|K^*K\|) \quad (k=0,1,2...)
\end{equation}
with some starting element $f_0 \in Q$, where $K^*$ denotes the adjoint operator of $K$. The continuous analog to (\ref{Landweber}) as $\Delta t$ tends to zero is known as {\it asymptotic regularization} or {\it Showalter's method} (see, e.g.,~\cite{Tautenhahn-1994,Vainikko1986}). It is in the form of a first order evolution equation
\begin{eqnarray}\label{FisrtFlow}
\dot{f}^\delta(t) + K^* K f^\delta(t)= K^* y^\delta, \quad f(0)= f_0,
\end{eqnarray}
where an artificial scalar time $t$ is introduced. There must be chosen an appropriate finite {\it stopping time} $T_*=T_*(\delta)$ (a priori choice) or $T_*=T_*(\delta,y^\delta)$ (a posteriori choice)
in order to ensure the regularizing property $f^\delta(T_*) \to f^\dagger$ as $\delta\to0$. Here and later on, $f^\dagger$ represents the unique minimum-norm solution of (\ref{OperatorEq}). Moreover, it has been shown that by using Runge-Kutta integrators, all of the properties of asymptotic regularization (\ref{FisrtFlow}) carry over to its numerical realization~\cite{Rieder-2005}. Hence, the continuous model (\ref{FisrtFlow}) is of particular importance for studying the intrinsic properties of a broad class of general regularization methods for inverse problems, and can be used for the development of new iterative regularization algorithms by combining some appropriate numerical schemes. Inspired by this, the authors in \cite{ZhangHof2018} studied the second order asymptotical regularization with the fixed damping parameter.

However, a fatal defect for large-scale problems is the slow performance of the Landweber iteration (too many iterations required for optimal stopping)  as well as of the (conventional and second order with a fixed damping parameter) asymptotical regularization methods, i.e.~overly excessive stopping times $T_*$ are required for obtaining optimal convergence rates. Therefore, in practice, accelerating strategies are usually used. In so doing, the most commonly known methods are the $\nu$-method \cite[\S~6.3]{engl1996regularization} and the Nesterov acceleration scheme \cite{Neubauer-2017}. Recently, the authors in \cite{ZhangHof2019} introduced the fractional order asymptotical regularization, and proved that the fractional order plays the role of acceleration. In this paper, we are interested in the following second order evolution equation with a linear vanishing damping term
\begin{eqnarray}\label{SecondFlow}
\ddot{f}^\delta(t) + \frac{1+2s}{t} \dot{f}^\delta(t) + K^* K f^\delta(t)= K^* y^\delta, \quad f(0)= f_0, ~\dot{f}(0)= 0,
\end{eqnarray}
where $s> -1/2$ is a fixed number. One motivation to study (\ref{SecondFlow}) is that it can be viewed as an infinite dimensional extension of the Nesterov's scheme in the sense that for all fixed $T > 0$ (\cite{Su-2016}): $\lim\limits_{\omega\to0} \max\limits_{0\leq k \leq T/\sqrt{\omega}} \|f^\delta_k - f^\delta(k\sqrt{\omega})\|_Q =0$, where $f^\delta(\cdot)$ is the dynamical solution of (\ref{SecondFlow}) with $s\geq1$, and $\{f^\delta_k\}_k$ is the sequence, generated by the Nesterov's scheme with parameters $(\alpha,\omega)$, see formula (\ref{Nesterov}) for details.

It should be noted that the second order dynamic (\ref{SecondFlow}) has recently been investigated in \cite{Botetal18}, where they have proven that the flow (\ref{OperatorEq}) with the vanishing initial data yields an optimal regularization method for the linear operator equation (\ref{OperatorEq}). In this paper, cf. Section 2, we focus on the acceleration effect in the sense of regularization theory of (\ref{OperatorEq}) with an arbitrary initial guess $f_0$. The main result of this paper regarding the discretized version of (\ref{SecondFlow}) is presented in Section 3, where we demonstrate that by using damped symplectic integrators, the regularization property and acceleration effect under optimal convergence rates of (\ref{SecondFlow}) carry over to its numerical realization. In Section~4, the developed accelerated iterative regularization methods, equipped with a posteriori stopping rule, are applied to a diffusion-based bioluminescence tomography, which can be formulated as:
\begin{problem}\label{prob:2.1}
Given $g_1$ and $g_2$, find a bioluminescent source $f$ such that the solution $u$ of the
boundary-value problem ($\Omega_0$ and $\Omega$ are bounded open domains in $\mathbb{R}^n$ ($n=2,3$))
\begin{eqnarray}\label{eq1}
\left\{\begin{array}{ll}
-{\rm div}(D\nabla u)+\mu_a u  = f \chi_{\Omega_0} \quad {\rm in\ }\Omega,  \\
D\partial_\nu u = g_2 \quad {\rm on}\ \Gamma
\end{array}\right.
\end{eqnarray}
satisfies
\begin{equation}
 u =g_1\quad {\rm on\ }\Gamma.\label{eq3}
\end{equation}
\end{problem}
Some numerical examples, as well as a comparison with three well-known existing iterative regularization methods, are presented in Section 5. Concluding remarks are given in Section 6. Some proof details as well as some details on finite element discretization are postponed to the appendices.

\section{Analysis of the continuous regularization method}\label{sec:method}

\subsection{Convergence analysis}\label{sec:regu}
We start with the convergence analysis of the continuous method (\ref{SecondFlow}) in the sense of regularization theory. Let $\{\lambda_j; u_j, v_j\}_{j=1}^\infty$ be the well-defined singular system for the compact linear operator $K$, i.e.~we have $K u_j= \lambda_j v_j$ and $K^* v_j = \lambda_j u_j$ with ordered singular values $\|K\|=\lambda_1 \geq \lambda_2 \geq \cdot\cdot\cdot \geq \lambda_j \geq \lambda_{j+1} \geq \cdot\cdot\cdot \to 0$ as $j \to \infty$. Since the eigenelements $\{u_j\}_{j=1}^\infty$ forms a orthogonal basis in $N(K)^{\bot} \subset Q$, any dynamical element $\hat{f}^\delta(t)\in Q$ has a decomposition $\hat{f}^\delta(t)=\sum_j \xi_j(t) u_j + \hat{f}_0(t)$, where $\hat{f}_0(t)\in N(K)$. As we are interested in an appropriate stable approximation of $f^\dagger$~-- the minimum norm least square solution of (\ref{OperatorEq}), the designed regularized solution can be chosen in the form $f^\delta(t)=\sum_j \xi_j(t) u_j$ by noting that $\|f^\delta(t)\|\leq \|\hat{f}^\delta(t)\|$ for all $\hat{f}_0(t)\in N(K)$. Since $f^\delta(t)$ satisfies the dynamic (\ref{SecondFlow}), we have
\begin{eqnarray}\label{SVDEq1}
\left\{\begin{array}{ll}
\langle \ddot{f}^\delta(t), u_j \rangle + \frac{1+2s}{t} \langle \dot{f}^\delta(t) , u_j \rangle + \lambda^2_j \langle f^\delta(t) , u_j \rangle = \lambda_j \langle y^\delta , v_j \rangle, ~j=1,2, \cdots\,  \\
f^\delta(0)= f_0,~\dot{f}^\delta(0)= 0,
\end{array}\right.
\end{eqnarray}
which implies the evolution equations for coefficient $\xi_j(t)$:
\begin{eqnarray}\label{SVDEq2}
\left\{\begin{array}{ll}
\ddot{\xi}_j(t) + \frac{1+2s}{t} \dot{\xi}_j(t) + \lambda^2_j \xi_j(t) = \lambda_j \langle y^\delta , v_j \rangle, \\
\xi_j(0)=\langle f_0,u_j \rangle, ~ \dot{\xi}_j(0)=0,
\end{array}\right. \quad j=1,2, \cdots.
\end{eqnarray}

\begin{proposition}\label{SolutionODE}
Let $s> -1/2$ be a fixed number. Then, the differential equation (\ref{SVDEq2}) has a unique solution
\begin{eqnarray*}\label{GeneralSolutionXi}
\xi_j(t) = 2^s\Gamma(s+1) \frac{J_s(\lambda_j t)}{(\lambda_j t)^s} \langle f_0,u_j \rangle + \left( 1- 2^s\Gamma(s+1) \frac{J_s(\lambda_j t)}{(\lambda_j t)^s} \right) \lambda^{-1}_j \langle y^\delta , v_j \rangle,
\end{eqnarray*}
where $\Gamma(\cdot)$ and $J_s(\cdot)$ denote the Gamma function and the Bessel function of first kind of order $s$ respectively.
\end{proposition}

The proof of the above proposition can be found in Appendix A. By Proposition \ref{SolutionODE} and the decomposition $f^\delta(t)=\sum_j \xi_j(t) u_j$ we obtain the explicit formula for the solution of (\ref{SVDEq1}) as
\begin{equation}\label{BVPsv}
\begin{array}{ll}
\hspace{-10mm} f^\delta(t) = \sum\limits_j  2^s\Gamma(s+1) \frac{J_s(\lambda_j t)}{(\lambda_j t)^s} \langle f_0,u_j \rangle  u_j + \sum\limits_j \left( 1- 2^s\Gamma(s+1) \frac{J_s(\lambda_j t)}{(\lambda_j t)^s} \right) \lambda^{-1}_j \langle y^\delta , v_j \rangle  u_j, \\
=: (1-K^* K g(t, K^* K))f_0 + g(t, K^* K) K^* y^\delta,
\end{array}
\end{equation}
where (we identify $\lambda^2_j$ as $\lambda$)
\begin{eqnarray}\label{gPhiDef}
g(t, \lambda) = \frac{1- 2^s\Gamma(s+1) \frac{J_s(\sqrt{\lambda} t)}{(\sqrt{\lambda} t)^s}}{\lambda}.
\end{eqnarray}

\begin{theorem}\label{RegularizationThm}
Let $f^\delta(t)$ be the dynamic solution of (\ref{SecondFlow}) with $s>-1/2$. Then, if the terminating time $T_*=T_*(\delta,y^\delta)$ is chosen so that
\begin{equation}\label{ReguParameters}
\lim_{\delta\to 0} T_* =\infty \textrm{~and~}  \lim_{\delta\to 0} \delta \cdot T_* = 0,
\end{equation}
the approximate solution $f^\delta(T_*)$ converges (strongly) to $f^\dagger$ as $\delta\to0$.
\end{theorem}

\begin{proof}
Let $f(t)$ be the solution of (\ref{SecondFlow}) with noise-free data, i.e., $f(t) = (1-K^* K g(t, K^* K))f_0 + g(t, K^* K) K^* y$. Furthermore, define the bias function by
\begin{eqnarray}\label{biasFun}
r(t, \lambda) = 1- \lambda g(t, \lambda) = 2^s\Gamma(s+1) \frac{J_s(\sqrt{\lambda} t)}{(\sqrt{\lambda} t)^s}.
\end{eqnarray}
Obviously, $r(t, \lambda)$ is the unique solution to
\begin{eqnarray}\label{Eq_r}
\left\{\begin{array}{ll}
\ddot{r}(t, \lambda) + \frac{1+2s}{t} \dot{r}(t, \lambda) + \lambda r(t, \lambda) = 0,  \\
r(0, \lambda)=1, \quad \dot{r}(0, \lambda)= 0.
\end{array}\right.
\end{eqnarray}
Then, with the help of the intermediate quantity $f(t)$ and bias function $r(t, \lambda)$, we obtain the well-known error estimates
\begin{eqnarray}\label{gPhiDef1}
\|f^\delta(t)-f^\dagger\| \leq \|f^\delta(t)-f(t)\| + \|f(t)-f^\dagger\| \nonumber \\ \qquad\qquad\qquad
\leq \delta \sup_{\lambda>0} \sqrt{\lambda} g(t, \lambda)  + \| r(t, K^* K) (f_0-f^\dagger)\|
\end{eqnarray}
by noting that $y=K f^\dagger$. Hence, to prove the convergence of the full regularization error, we have to show the convergence of both two terms in the right-hand side of (\ref{gPhiDef1}).

Let's first consider the estimate for $\| r(t, K^* K) (f_0-f^\dagger)\|$. To this end, define the Lyapunov function of (\ref{Eq_r}) by $E(t):= \dot{r}^2(t, \lambda) + \lambda r^2(t, \lambda)$. Since
\begin{eqnarray*}
\dot{E}(t) = 2 \dot{r}(t, \lambda) [\ddot{r}(t, \lambda)+ \lambda r(t, \lambda)] = -\frac{2(1+2s)}{t} \dot{r}^2(t, \lambda) \leq 0,
\end{eqnarray*}
$E(t)$ is a non-increasing function, and consequently, we have $\lambda r^2(t, \lambda)\leq E(t)\leq E(0)=\lambda$, which implies
\begin{eqnarray}\label{bounded_r}
|r(t, \lambda)| \leq 1 \textrm{~for all~} \lambda>0, t\geq0.
\end{eqnarray}

On the other hand, by asymptotic \cite[(9.2.1)]{Abramowitz1972}
\begin{eqnarray}\label{AsymptoticJ0}
J_{s}(\sqrt{\lambda} t)= \sqrt{\frac{2}{\pi}} \left(\sqrt{\lambda}t \right)^{-\frac{1}{2}} \cos\left(\sqrt{\lambda} t - \frac{\pi(2s+1)}{4} \right) + \mathcal{O}\left( \frac{1}{\sqrt{\lambda} t} \right) \textrm{~as~} t\to \infty,
\end{eqnarray}
for any fixed $\lambda>0$, we obtain together with (\ref{bounded_r}) and condition $s>-1/2$ that
\begin{equation}\label{LimitRegu1}
 \| r(t, K^* K) (f_0-f^\dagger)\| \to 0 \textrm{~as~} t\to \infty.
\end{equation}

Now, consider the quality $\sqrt{\lambda} g(t, \lambda)$. To this end, we introduce the function $\zeta(\tau)= \zeta(\sqrt{\lambda}t)=r(t, \lambda)$. By the representation of the Bessel functions, see e.g. \cite[(9.5.10)]{Abramowitz1972}, we have $\zeta(\tau)=\prod^{\infty}_{k=1} \left( 1 - \frac{\tau^2}{\jmath^2_{s,k}} \right)$, where $\jmath_{s,k}$ denotes the $k$-th positive zero of Bessel functions $J_s$ (sorted in increasing order). Hence, for all $\tau\in (0, \jmath_{s,1})$,
\begin{equation*}
\zeta'(\tau)= -2 \tau \sum^{\infty}_{k=1} \left\{ \jmath^{-2}_{s,k} \prod^{\infty}_{i\neq k} \left( 1 - \frac{\tau^2}{\jmath^2_{s,i}} \right) \right\} < 0,
\end{equation*}
which means that $\zeta$ is monotonically decreasing on $(0, \jmath_{s,1})$. According to the the initial condition (\ref{Eq_r}), we have $\zeta(0)=1$ and $\zeta'(0)=0$, which implies together with the asymptotic (\ref{Jlimit}) that a number $\tau_0 \in (0,\jmath_{s,1})$ exists such that for all $\tau\in[0,\tau_0]$: $\zeta(\tau) \geq 1- \tau$. Setting $\tau_1 :=\min \left\{1/2, \tau_0 /4 \right\}$, we obtain
\begin{equation}\label{lowerBoundr}
\zeta(\tau)\geq 1- \tau \geq 1- \frac{\tau}{2\tau_1}, \textrm{~for~} \tau\in(0,\tau_0].
\end{equation}
On the other hand, by (\ref{bounded_r}), we deduce that
\begin{equation}\label{lowerBoundr2}
\zeta(\tau)=r(t, \lambda)\geq -1 \geq 1- \frac{2}{\tau_0} \tau \geq 1- \frac{\tau}{2\tau_1} \textrm{~for~} \tau\geq \tau_0.
\end{equation}
Combine (\ref{lowerBoundr}), (\ref{lowerBoundr2}) and the relations $\zeta(\tau)= \zeta(\sqrt{\lambda}t)=r(t, \lambda)$ to obtain
\begin{equation*}
r(t, \lambda)\geq 1- \frac{\sqrt{\lambda}t}{2\tau_1} \textrm{~for all~} \lambda>0, t\geq0.
\end{equation*}
Consequently, we have
\begin{eqnarray}\label{Ineq_g}
\sqrt{\lambda} g(t, \lambda) = \frac{1 - r(t, \lambda)}{\sqrt{\lambda}} \leq \frac{t}{2\tau_1},
\end{eqnarray}
which implies that $\delta \sup_{\lambda>0} \sqrt{\lambda} g(t, \lambda) \leq \delta \frac{t}{2\tau_1} \to 0$ under the choice of terminating time in (\ref{ReguParameters}).

\end{proof}

\subsection{Convergence rate and acceleration}\label{sec:acceleration}

The purpose of this subsection is to show that (\ref{SecondFlow}) with an appropriate terminating time yields an accelerated optimal regularization method. It is well-known that in order to prove the convergence rate for the approximate solution $f^\delta(t)$, additional smoothness assumptions on $f^\dagger$ in correspondence with the forward operator $K$ have to be fulfilled. For simplicity, in this paper, we only consider the H\"odler type source conditions, i.e.
\begin{assumption}\label{SourceConditionAssumption}
There exists an element $v_0$ and a number $\rho\geq0$ such that
\begin{equation}\label{SourceCondition}
f_0 - f^\dagger = \left( K^* K \right)^{\mu} v_0 \quad \textrm{~with~} \quad \|v_0\|\leq \rho.
\end{equation}
\end{assumption}
By using the technique of the comparison of two qualifications, cf. \cite[Def.~2]{Mathe-2003} and \cite[Prop.~3, Remark~5 and Lemma~2]{Mathe-2003}, the results in this subsection can be easily extended to general range-type source conditions such as the logarithmic source conditions.

\begin{theorem}\label{ThmPriori}
(A priori choice of the terminating time) \\ Let $f^\delta(t)$ be the solution of (\ref{SecondFlow}) with $s>-1/2$. Then, under the Assumption \ref{SourceConditionAssumption},
\begin{itemize}
\item[(i)] if $\mu\in (0, \frac{1+2s}{4}]$ and $T_* = \delta^{-\frac{1}{2\mu+1}}$, we have the order optimal convergence rate
\begin{equation}\label{ErrorEstimatePriori}
\| f^\delta(T_*) - f^\dagger \| =\mathcal{O}\left( \delta^{\frac{2\mu}{2\mu+1}} \right) \textrm{~as~} \delta\to 0.
\end{equation}
\item[(ii)] if $\mu> \frac{1+2s}{4}$ and $T_* = \delta^{-\frac{2}{2s+3}}$, we have the reduced convergence rate
\begin{equation}\label{ErrorEstimatePriori2}
\| f^\delta(T_*) - f^\dagger \| =\mathcal{O}\left( \delta^{\frac{2s+1}{2s+3}} \right) \textrm{~as~} \delta\to 0.
\end{equation}
\end{itemize}
\end{theorem}

\begin{proof}
According to (\ref{AsymptoticJ0}), there exists a pair of numbers $(C_0,T_0)$ such that for all $t\geq T_0$: $J_{r}(\sqrt{\lambda} t) \leq C_0 \lambda^{-1/4} t^{-1/2}$, which implies together with (\ref{biasFun}), (\ref{bounded_r}), and Assumption \ref{SourceConditionAssumption} that
\begin{eqnarray}\label{Estimate1}
\begin{array}{ll}
\| r(t, K^* K) (f_0-f^\dagger)\| =  \| r(t, K^* K) \left( K^* K \right)^{\mu} v_0 \| \leq \rho \sup_{\lambda} r(t, \lambda) \lambda^\mu \\ \quad
\leq C_0 2^s\Gamma(s+1) \rho \sup_{\lambda\in (0, \|K\|^2]} \min\left\{ \lambda^{-\frac{1+2s}{4}} t^{-\frac{1+2s}{2}}, 1 \right\} \lambda^\mu \\ \quad
\leq C_1 \cdot \left\{\begin{array}{l}
t^{-2\mu}, \textrm{~if~} \mu\in (0, \frac{1+2s}{4}] \\
t^{-\frac{1+2s}{2}}, \textrm{~if~} \mu> \frac{1+2s}{4},
\end{array}\right.
\end{array}
\end{eqnarray}
where $C_1= C_0 2^s\Gamma(s+1) \rho \max\left\{ \|K\|^{2\mu- \frac{1+2s}{2}},1 \right\}$. We complete the proof by the following inequalities
\begin{eqnarray}\label{IneqFinal}
\begin{array}{ll}
\|f^\delta(T_*)-f^\dagger\|
\leq \left\{\begin{array}{l}
\delta \frac{T_*}{2\tau_1} + C_1 (T_*)^{-2\mu} = (\frac{1}{2\tau_1}+ C_1) \delta^{\frac{2\mu}{2\mu+1}}, \textrm{~if~} \mu\in (0, \frac{1+2s}{4}] \\
\delta \frac{T_*}{2\tau_1} + C_1 (T_*)^{-\frac{1+2s}{2}} = (\frac{1}{2\tau_1}+ C_1) \delta^{\frac{2s+1}{2s+3}}, \textrm{~if~} \mu> \frac{1+2s}{4},
\end{array}\right.
\end{array}
\end{eqnarray}
due to (\ref{gPhiDef}), (\ref{Ineq_g}), (\ref{Estimate1}) and the choice of $T_*$.
\end{proof}

\begin{remark}
Denote by $\theta=\frac{4\mu}{1+2s}$. Then, assertion (ii) of Theorem \ref{ThmPriori} can be reformulated as follows: in the case $\mu> \frac{1+2s}{4}$ (or equivalently, $\theta>1$), if the terminating time is chosen by $T_* = \delta^{-\frac{\theta}{2\mu+\theta}}$, we have the convergence rate $\| f^\delta(T_*) - f^\dagger \| =\mathcal{O}\left( \delta^{\frac{2\mu}{2\mu+\theta}} \right)$ as $\delta\to 0$. This means that the small choice of model parameter $s$ will reduce the accuracy of the estimated approximate solution of our regularization method (\ref{SecondFlow}). However, if the value of $s$ is not small, i.e. $\mu<s+1$, (\ref{SecondFlow}) still offers an accelerated regularization method by noting that $T_* = \delta^{-\frac{\theta}{2\mu+\theta}} < T_{Landweber} = \delta^{-\frac{2}{2\mu+1}}$.
\end{remark}

The following proposition indicates that $\mu=\frac{1+2s}{4}$ cannot be a qualification in the sense of regularization theory \cite{engl1996regularization}.

\begin{proposition}\label{qualificationSharp}
Let $\mu\in(0,(1+s)/2)$. Then $\mu$ cannot be a qualification of method (\ref{SecondFlow}) under source conditions (\ref{SourceCondition}).
\end{proposition}

\begin{proof}
We prove it by contradiction. According to \cite[\S~4.2]{engl1996regularization}, it is necessary to show that for any positive constants $c, \gamma$ there exists numbers $T_1$ and $\lambda \in [ct^{-2}, \|K\|^2]$ such that the inequality
\begin{equation}\label{ConverseIneq}
(\lambda t^2)^{\mu} |r(t,\lambda)| \geq \gamma
\end{equation}
cannot be hold for all $t\geq T_1$. Indeed, according to the asymptotic (\ref{AsymptoticJ0}), there exist two constants $C$ and $T_0$ such that for all $t\geq T_0$,
\begin{equation}\label{IneqJ}
|J_{s}(\sqrt{\lambda} t)| \leq \left| \sqrt{\frac{2}{\pi}} \left(\sqrt{\lambda}t \right)^{-\frac{1}{2}} \cos\left(\sqrt{\lambda} t - \frac{\pi(2s+1)}{4} \right) \right| + \frac{C}{\sqrt{\lambda} t}.
\end{equation}
On the other hand, for
\begin{equation*}
T_1 := \|K\|^{-1} \max\left\{ \sqrt{2} [C 2^s \Gamma(s+1)\gamma^{-1}]^{1/(1+s-2\mu)} ,  \sqrt{c+\pi/2}, T_0 \right\},
\end{equation*}
there must exist an integer $k_1> \pi^{-1} (C2^s \Gamma(s+1)\gamma^{-1})^{2/(1+s-2\mu)} -1/2$ such that $k_1 \pi + \pi/2 \in [c, \|K\|^2 t^2]$ for any $t\geq T_1$, and for $\lambda_1:=(k_1 \pi + \pi/2)t^{-2} \in [ct^{-2}, \|K\|^2]$ (according to (\ref{IneqJ})),
\begin{equation*}
\hspace{-10mm} (\lambda_1 t^2)^{\mu} |r(t,\lambda_1)| \leq C2^s\Gamma(s+1) (\sqrt{\lambda_1} t)^{2\mu-(1+s)} =  C2^s\Gamma(s+1) (k_1 \pi + \pi/2)^{\mu-(1+s)/2} < \gamma,
\end{equation*}
which contradicts the inequlaity (\ref{ConverseIneq}).
\end{proof}

Now, let us turn to the a posteriori choice of the terminating time $T_*=T_*(\delta,y^\delta)$. We consider Morozov's discrepancy
principle as the most prominent version which exploits zeros of the discrepancy function
\begin{eqnarray}\label{discrepancy1}
\chi(t):= \|K f^\delta(t)-y^\delta\| - \tau \delta,
\end{eqnarray}
where $\tau> 1$ is a fixed parameter.

\begin{lemma}\label{Rootdiscrepancy}
If $\|K f_0-y^\delta\|> \tau \delta$, then $\chi(T)$ has at least one solution.
\end{lemma}

The proof of the above lemma can be found in Appendix B. If the function $\chi(T)$ has more than one root, we recommend selecting $T_*$ from the rule
\begin{eqnarray*}
\chi(T_*) = 0< \chi(T), \quad \forall T<T_*.
\end{eqnarray*}
In other words, $T_*$ is the first time point for which the size of the residual $\|K f^\delta(t)-y^\delta\|$ has about the order of the data error. By Lemma~\ref{Rootdiscrepancy} such $T_*$ always exists. Furthermore, by the proof of Lemma \ref{Rootdiscrepancy} in Appendix B, it is easy to show that $\chi(T)$ is bounded by a monotonically decreasing function $\mathcal{E}(t)- \tau \delta \searrow - \tau \delta$ as $t\to \infty$. Hence, roughly speaking,  the trend of $\chi(T)$ is to be a decreasing function, where oscillations may occur.

\begin{theorem}\label{ThmPosteriori}
(A posteriori choice of the terminating time) Let the terminating time of $f^\delta(T_*)$ of (\ref{SecondFlow}) be chosen as a root of the discrepancy function (\ref{discrepancy1}). Then, under Assumption \ref{SourceConditionAssumption},
\begin{itemize}
\item[(i)] if $\mu\in (0, \frac{2s-1}{4}]$, we have the order optimal error estimates
\begin{equation}\label{ErrorEstimatePrioriT}
T_* = \mathcal{O} \left( \delta^{-\frac{1}{2\mu+1}} \right), \quad \| f^\delta(T_*) - f^\dagger \| = \mathcal{O} \left( \delta^{\frac{2\mu}{2\mu+1}} \right) \textrm{~as~} \delta\to 0.
\end{equation}
\item[(ii)] If $\mu> \frac{2s-1}{4}$, we have the reduced error estimates
\begin{equation}\label{ErrorEstimatePrioriT2}
T_* = \mathcal{O} \left( \delta^{-\frac{2}{2s+3}} \right), \quad \| f^\delta(T_*) - f^\dagger \| = \mathcal{O} \left( \delta^{\frac{2s+1}{2s+3}} \right) \textrm{~as~} \delta\to 0.
\end{equation}
\end{itemize}
\end{theorem}

The proof of Theorem \ref{ThmPosteriori} can be done by a standard argument in regularization theory, and we omit it here. By Theorems \ref{ThmPriori} and \ref{ThmPosteriori}, for the proposed continuous regularization method (\ref{SecondFlow}) the optimal convergence rates can be obtained with approximately the square root of iterations than would be needed for the conventional asymptotical regularization method \cite[\S~6.2]{engl1996regularization}, which means that our method is an accelerated regularization method. However, similar to the existing accelerated order-optimal regularization methods (e.g. $\nu$-method \cite[\S~6.3]{engl1996regularization}, Nesterov's method \cite{Neubauer-2017}, and fractional asymptotical regularization \cite{ZhangHof2019}), the proposed method (\ref{SecondFlow}) also shows a saturation phenomenon; i.e.,
the optimal convergence rate $\| f^\delta(T_*) - f^\dagger \| = \mathcal{O} \left( \delta^{\frac{2\mu}{2\mu+1}} \right)$ and the asymptotic $T_* = \mathcal{O} \left( \delta^{-\frac{1}{2\mu+1}} \right)$ holds only for $\mu\in (0, \frac{2s+1}{4}]$ or $\mu\in (0, \frac{2s-1}{4}]$ according to the choice of the terminating time. Therefore, the choice of $s$ is crucial for applying the regularization method (\ref{SecondFlow}). The a priori knowledge of the degree of smoothness of unknown exact solution provides a low bound for the model parameter $s$. It should be noted that, though theoretically, the large value of $s$ can extend the region of optimal convergence rate and increase the convergence rate in the case of high smooth exact solution, in practice, too large value of $s$ may decrease the accuracy of the obtained approximate solution since the constant in the asymptotic $\mathcal{O} \left( \delta^{\frac{2\mu}{2\mu+1}} \right)$, e.g. $C_1$ in (\ref{Estimate1}), blows up as $s\to\infty$. A detailed discussion will be given in Section 5.1.

\section{A new class of accelerated iterative regularization methods}\label{sec:class}

The evolution equation (\ref{SecondFlow}) with an appropriate numerical discretization scheme for the artificial time variable yields a concrete iterative method. This has motivated us to develop some novel iterative regularization methods based on the continuous method (\ref{SecondFlow}). The goal of this section is to realize this idea.

Just as with the Runge-Kutta integrators~\cite{Rieder-2005} or the exponential integrators~\cite{Hochbruck-1998} for numerically solving first order equations, the damped symplectic integrators are extremely attractive for solving (\ref{SecondFlow}), since the schemes are closely related to the canonical transformations~\cite{Hairer-2006}, and the trajectories of the discretized second flow are usually more stable for its long-term performance. To this end, let us start with the simplest symplectic scheme~-- the symplectic Euler method, i.e.
\begin{equation}\label{symplecticEuler2}
\left\{\begin{array}{l}
q^{k+1} = q^{k} + \Delta t_k \left( K^*(y^\delta-K f^{k}) - \frac{1+2s}{t_{k}} q^{k} \right) , \\
f^{k+1} = f^{k} + \Delta t_{k+1} q^{k+1} ,
\end{array}\right.
\end{equation}
By elementary calculations, scheme (\ref{symplecticEuler2}) can express in the form of following three-term semi-iterative method
\begin{equation}\label{semiiterative}
f^{k+1}=f^{k} + a_k \left( f^{k}-f^{k-1} \right) + \omega_k K^*(y^\delta-K f^{k})
\end{equation}
with parameters $a_k = \frac{\Delta t_{k+1}}{\Delta t_k} \left( 1-  \Delta t_k \frac{1+2s}{t_{k}} \right)$ and $\omega_k= \Delta t_{k} \Delta t_{k+1}$. The adjoint scheme of (\ref{symplecticEuler2}), namely
\begin{equation}\label{symplecticEuler}
\left\{\begin{array}{l}
f^{k+1} = f^{k} + \Delta t_k q^{k} , \\
q^{k+1} = q^{k} + \Delta t_{k+1} \left( K^*(y^\delta-K f^{k+1}) - \frac{1+2s}{t_{k+1}} q^{k} \right),
\end{array}\right.
\end{equation}
also shares the same recurrence form (\ref{semiiterative}), but with parameters $a_k = \frac{\Delta t_{k}}{\Delta t_{k-1}} \left( 1-  \Delta t_{k-1} \frac{1+2s}{t_{k}} \right)$ and $\omega_k= \Delta t^2_{k}$. Obviously, when we fix the step size $\Delta t_k\equiv\Delta t$, iterations (\ref{symplecticEuler2}) and (\ref{symplecticEuler}) coincide with each other . For the high order symplectic methods, we can consider the St\"{o}rmer-Verlet scheme (with a constant step size $\Delta t$), which takes the form
\begin{equation}\label{symplectic0}
\left\{\begin{array}{l}
q^{k+\frac{1}{2}} = q^{k} - \frac{\Delta t}{2}  \frac{1+2s}{t_k} q^{k+\frac{1}{2}} + \frac{\Delta t}{2} K^*(y^\delta-K f^{k}), \\
f^{k+1} = f^{k} + \Delta t q^{k+\frac{1}{2}} , \\
q^{k+1} = q^{k+\frac{1}{2}} - \frac{\Delta t}{2}\frac{1+2s}{t_{k+1}} q^{k+\frac{1}{2}} + \frac{\Delta t}{2} K^*(y^\delta-K f^{k+1}), \\
\end{array}\right.
\end{equation}
Surprisedly, the scheme (\ref{symplectic0}) can also be rewritten in the form of (\ref{semiiterative}), with parameters
\begin{equation*}
a_k = \frac{1- \frac{\Delta t(1+2s)}{2 t_{k}}}{1+ \frac{\Delta t(1+2s)}{2 t_{k}}}, \quad
\omega_k=\frac{\Delta t^2}{1+ \frac{\Delta t(1+2s)}{2 t_{k}}}.
\end{equation*}

In the work, we consider the following modified St\"{o}rmer-Verlet scheme
\begin{equation}\label{symplectic}
\left\{\begin{array}{l}
q^{k+\frac{1}{2}} = q^{k} - \frac{\Delta t}{2}  \frac{1+2s}{t_{k}} q^{k+\frac{1}{2}} + \frac{\Delta t}{2} K^*(y^\delta-K f^{k}), \\
f^{k+1} = f^{k} + \Delta t q^{k+\frac{1}{2}} , \\
v^{k+1} = f^{k+1} + 2 \Delta t a_{k+1} q^{k+\frac{1}{2}} , \\
q^{k+1} = q^{k+\frac{1}{2}} - \frac{\Delta t}{2}\frac{1+2s}{t_{k+1}} q^{k+\frac{1}{2}} + \frac{\Delta t}{2} K^*(y^\delta-K v^{k+1}), \\
\end{array}\right.
\end{equation}
where the third step in (\ref{symplectic}) is inspired by the Nesterov's method. Here, parameters $a_k$ will be defined later in (\ref{a_k}). The goal in this section is to prove that the scheme (\ref{symplectic}) with a posteriori iteration stopping rule yields an accelerated iterative regularization method.

\begin{remark}
(a) At the beginning of iteration in (\ref{symplectic}), one can set $t_0:=t_1$ to avoid the singularity. (b) By adding $v^{k+1} = f^{k+1} + 2\Delta t a_{k+1} q^{k+\frac{1}{2}}$ in the symplectic Euler method (\ref{symplectic0}), one can obtain a new iteration scheme with a lower computational cost in the inner iterations. The convergence analysis is exactly the same as for the scheme (\ref{symplectic}). The behaviour of this method is similar to the Nesterov's method, and thus, we omit it here.
\end{remark}

Unlike the original the symplectic Euler method and St\"{o}rmer-Verlet, the scheme (\ref{symplectic}) expresses the following recurrence form
\begin{equation}\label{recurrenceOur}
\hspace{-15mm} f^{k+1}=f^{k} + a_k \left( f^{k}-f^{k-1} \right) + \omega_k K^*\left(y^\delta-K \left(f^{k} + a_k \left( f^{k}-f^{k-1} \right) \right) \right), ~k=1,2,\cdots,
\end{equation}
with parameters
\begin{equation}\label{a_k}
a_k = \frac{1- \frac{\Delta t(1+2s)}{2 t_{k}}}{1+ \frac{\Delta t(1+2s)}{2 t_{k}}} = \frac{2k - (1+2s)}{2k + (1+2s)}, \quad
\omega_k=\frac{\Delta t^2}{1+ \frac{\Delta t(1+2s)}{2 t_{k}}} = \frac{2\Delta t^2 k}{2k + 1+2s}.
\end{equation}

Without loss of generality, for $s>1/2$ define that
\begin{equation}\label{omega_k}
\omega_k :=\frac{\Delta t^2}{2} \quad \textrm{~for~} \quad k< \max\left\{s+\frac{1}{2},k_s \right\},~ k_s:= \left\lceil \frac{1}{2} + \frac{1}{2s-1} \right\rceil.
\end{equation}
Consequently, $\omega_k \geq \Delta t^2/2$ for all $k\in \mathbb{N}$.

As the Landweber iterates, according to (\ref{recurrenceOur}), the iterates $f^{k}$ of (\ref{symplectic}) obviously belong to the Krylov subspace $\textrm{Span}\left\{ K^* y^\delta, \cdots, (K^*K)^{k-1} K^* y^\delta \right\}$. Therefore, the solution $f^{k}$ of (\ref{symplectic}) can be written as $f^{k}=g_k(K^*K)K^* y^\delta$, where $g_k$ is a polynomial of degree $k - 1$, and the residual polynomials
\begin{equation}\label{BiasFun}
r_k(\lambda)= 1- \lambda g_k(\lambda)
\end{equation}
exhibit the following property.

\begin{proposition}\label{BiasFunIneq}
Assume that $s>1/2$ and $\Delta t\in (0, \sqrt{2}/\|K\|) $. Then, the residual polynomials of scheme (\ref{symplectic}) satisfy the following inequality
\begin{equation}\label{BiasFunIneq1}
\sup_{\lambda\in(0,\|K\|^2]} \lambda^{\mu} r_k(\lambda) \leq c_1 k^{-2\mu}, \quad \textrm{~if~} \mu\in (0,1/2],
\end{equation}
where $c_1=2^{3\mu} s^{2\mu} \Delta t^{-2\mu}$ and
\begin{equation}\label{BiasFunIneq2}
\sup_{\lambda\in(0,\|K\|^2]} \lambda^{\mu} r_k(\lambda) \leq c_2 k^{-(\mu+1/2)}, \quad \textrm{~if~} \mu> 1/2,
\end{equation}
where $c_2= 2^{3\mu} \Delta t^{-2\mu} (\mu-1/2)^{\mu-1/2} s$.
\end{proposition}

\begin{proof}
This proof uses the technique in \cite{Neubauer-2017}. By using (\ref{recurrenceOur}) and elementary calculations, the residual polynomials of (\ref{symplectic}) satisfy the recurrence relation
\begin{equation*}
r_{k+1} = \left( 1- \omega_k\lambda \right) \left[ r_{k} + a_k \left( r_{k} - r_{k-1} \right) \right],
\end{equation*}
which can be rewritten as
\begin{equation}\label{recurrence_r}
r_{k+1} = \left( 1- \omega_{k} \lambda \right) \left[ (1-\theta_k) r_{k} +  \theta_k \left( r_{k-1} + \frac{1}{\theta_{k-1}} ( r_{k} - r_{k-1} ) \right)  \right],
\end{equation}
where $a_{k} = \frac{\theta_k}{\theta_{k-1}} (1-\theta_{k-1})$, and
\begin{equation}\label{theta_k}
\theta_k = \frac{4s}{2k+1+2s}.
\end{equation}

Now, let us show that the sequence $(\theta_k)_k$, defined in (\ref{theta_k}), satisfies the following inequality
\begin{equation}\label{IneqTheta}
\omega_{k} \frac{(1-\theta_k)^2}{\theta^2_k} \leq \frac{\omega_{k-1}}{\theta^2_{k-1}} \textrm{~for~} k>1.
\end{equation}

By definitions of $\omega_{k}$ and $\theta_k$ in (\ref{a_k}), (\ref{omega_k}) and (\ref{theta_k}), we have under the assumption $s>1/2$ that
\begin{equation}\label{theta_k2}
\hspace{-10mm} 1 \leq \left(  \sqrt{\frac{\omega_{k}}{\omega_{k-1}}} \leq \frac{\omega_{k}}{\omega_{k-1}} = \frac{k(2k-1+2s)}{(k-1)(2k+1+2s)} \leq \right) \frac{2k-1+2s}{2k+1-2s} = \frac{\theta_k}{1-\theta_k} \frac{1}{\theta_{k-1}},
\end{equation}
which yields the inequality (\ref{IneqTheta}). The inequalities in the bracket of (\ref{theta_k2}) are used when $k\geq k_s$.

Now, denote by $\hat{r}_{k} := r_{k-1} + \frac{1}{\theta_{k-1}} ( r_{k} - r_{k-1} )$. Then, we derive together with (\ref{recurrence_r}) and (\ref{IneqTheta}) that
\begin{equation}\label{PfDisIneq1}
\begin{array}{l}
\frac{\omega_{k}}{\theta^2_k} \lambda r^2_{k+1} + \left( 1- \omega_{k}\lambda \right) \hat{r}^2_{k+1} = \frac{\omega_{k}}{\theta^2_k} \lambda r^2_{k+1} + \left( 1- \omega_{k}\lambda \right) \left[ \frac{r_{k+1}}{\theta_k} - \frac{1-\theta_k}{\theta_k} r_{k} \right]^2
\\ \qquad = \frac{r^2_{k+1}}{\theta^2_k}  + \left( 1- \omega_{k}\lambda \right) \left[ \left( \frac{1-\theta_k}{\theta_k} \right)^2 r^2_{k} - \frac{2(1-\theta_k)}{\theta^2_k} r_{k} r_{k+1}  \right]
\\ \qquad = \left( 1- \omega_{k}\lambda \right) \left[ \frac{(1-\theta_k)^2}{\theta^2_k}  \omega_{k}\lambda r^2_{k} +  \left( 1- \omega_{k}\lambda \right) \hat{r}^2_{k} \right] \\ \qquad
\leq \left( 1- \omega_{k}\lambda \right) \left[ \frac{\omega_{k-1}}{\theta^2_{k-1}}  \lambda r^2_{k} +  \left( 1- \omega_{k-1}\lambda \right) \hat{r}^2_{k} \right] \\ \qquad
\leq  \left[ \prod^{k}_{i=1} \left( 1- \omega_{i} \lambda \right) \right] \cdot \left[ \frac{\omega_1}{\theta^2_{0}}  \lambda r^2_{1} +  \left( 1- \omega_1\lambda \right) \hat{r}^2_{1} \right] = \prod^{k}_{i=1} \left( 1- \omega_{i}\lambda \right)
\end{array}
\end{equation}
by noting that $\theta_0=\hat{r}_{1}=1$ (as $r_{-1}=r_{0}\equiv 1$). Inequality (\ref{PfDisIneq1}) immediately yields
\begin{equation}\label{PfDisIneq2}
\lambda r^2_{k+1} \leq \frac{\theta^2_k}{\omega_{k}} \prod^{k}_{i=1} \left( 1- \omega_{i} \lambda \right) \leq \frac{2 \theta^2_k}{\Delta t^2} \left( 1- \frac{\Delta t^2}{2} \lambda \right)^{k} 
\end{equation}
as well as $|\hat{r}_{k+1} | \leq 1$. The latter inequality together with the recurrence (\ref{recurrence_r}) and initial data $r_{-1}=r_{0}\equiv 1$ implies
\begin{equation}\label{PfDisIneq3}
| r_{k}| \leq 1 \textrm{~for all~} k\geq1.
\end{equation}

If $\mu\in (0,1/2]$, (\ref{PfDisIneq2}), (\ref{PfDisIneq3}) and the definition of $\theta_k$ in (\ref{theta_k}) immediately gives
\begin{equation*}
\lambda^{\mu} r_k(\lambda) \leq \left( \lambda r^{2}_k(\lambda) \right)^{\mu}  r^{1-2\mu}_k(\lambda) \leq \left( \frac{2 \theta^2_{k-1}}{\Delta t^2} \right)^\mu \leq c_1 k^{-2\mu}.
\end{equation*}

If $\mu>1/2$, we obtain together with (\ref{PfDisIneq2}) that (let $k\geq2$)
\begin{equation*}
\begin{array}{l}
\lambda^{\mu} r_k(\lambda) = \lambda^{\frac{1}{2}} r_k(\lambda) \lambda^{\mu-\frac{1}{2}} \leq \frac{\sqrt{2} \theta_{k-1}}{\Delta t} \left( 1- \frac{\Delta t^2}{2} \lambda \right)^{\frac{k-1}{2}} \lambda^{\mu-\frac{1}{2}} \\
\underset{\lambda_{max} = \frac{4\mu-2}{\Delta t^2(2\mu+k-2)}}{\leq}  \frac{\sqrt{2} \theta_{k-1}}{\Delta t} \left( \frac{k-1}{k+2\mu-2} \right)^{\frac{k-1}{2}} \left( \frac{4\mu-2}{\Delta t^2(2\mu+k-2)} \right)^{\mu-\frac{1}{2}} \leq c_2 k^{-\left(\mu+\frac{1}{2}\right)}.
\end{array}
\end{equation*}

\end{proof}

By inequality (\ref{PfDisIneq2}), it is not difficult to show that the following limit
\begin{equation}\label{Limit_r}
\lambda^{\mu} |r_k(\lambda)| \to 0 \textrm{~as~} k\to \infty
\end{equation}
holds for all fixed $\lambda\in(0,\|K\|^2]$ and $\mu\geq0$. Then, Based on the relation (\ref{Limit_r}), Proposition \ref{BiasFunIneq} and standard argument for linear regularization theory, see e.g. \cite[Theorems 3.1 and 4.1]{Neubauer-2017} or \cite{engl1996regularization}, we have the following convergence rate results.

\begin{theorem}\label{ThmDis}
Suppose that $s>1/2$ and $\Delta t\in (0, \sqrt{2}/\|K\|)$. Let $f^{k}$ be the approximate solution, generated by the scheme (\ref{symplectic}). Then, under Assumption \ref{SourceConditionAssumption},
\begin{itemize}
\item if $\mu\in (0,1/2]$ and $k^*= \mathcal{O} ( \delta^{-\frac{1}{2\mu+1}} )$, we have the convergence rate
\begin{equation}\label{ErrorEstimatePrioriDis}
\| f^{k^*} - f^\dagger \| = \mathcal{O} \left( \delta^{\frac{2\mu}{2\mu+1}} \right) \textrm{~as~} \delta\to 0.
\end{equation}

\item If $\mu>1/2$ and $k^*= \mathcal{O} ( \delta^{-\frac{1}{2\mu+3}} )$, we have the convergence rate
\begin{equation}\label{ErrorEstimatePrioriDis2}
\| f^{k^*} - f^\dagger \| = \mathcal{O} \left( \delta^{\frac{2\mu+1}{2\mu+3}} \right) \textrm{~as~} \delta\to 0.
\end{equation}

\item For general positive $\mu$, if the iteration of (\ref{symplectic}) is terminated according to the discrepancy principle (with a fixed positive parameter $\tau$), i.e.
\begin{equation}\label{discrepancy2}
\|y^\delta-K f^{k^*}\| \leq \tau \delta < \|y^\delta-K f^{k}\|, \quad 0\leq k < k^*,
\end{equation}
then, it holds that
\begin{equation}\label{ErrorEstimatePrioriDis3}
k^*= \mathcal{O} \left( \delta^{-\frac{1}{\mu+1}} \right), \quad
\| f^\delta(k^*) - f^\dagger \| = o\left( \delta^{\frac{\mu}{\mu+1}} \right) \textrm{~as~} \delta\to 0.
\end{equation}
\end{itemize}
\end{theorem}

We end this section by offering a few remarks.

\begin{remark}
Unlike in the continuous situation, the a priori stopping rules in the first two assertions are not optimal. Consequently, the convergence rates in (\ref{ErrorEstimatePrioriDis}) and (\ref{ErrorEstimatePrioriDis2}) are not optimal for our discretized regularization method (\ref{symplectic}). Indeed, similar to \cite[Theorems 3.1]{Neubauer-2017}, one can show that under the following a priori stopping rule (Obviously, it is not realizable, since $f^\dagger$ is not known):
\begin{equation}\label{prioriOptimal}
\frac{\|f^{k^*} - f^\dagger\|}{k^*} \leq \tau_0 \delta  \leq \frac{\|f^{k} - f^\dagger\|}{k}, \quad k<k^*(\delta), \tau_0>0,
\end{equation}
if $\mu\in (0,1/2]$, we have
\begin{equation*}
k^*= \mathcal{O} \left( \delta^{-\frac{1}{2\mu+1}} \right), \quad
\| f^\delta(k^*) - f^\dagger \| = o\left( \delta^{\frac{2\mu}{2\mu+1}} \right) \textrm{~as~} \delta\to 0.
\end{equation*}
If $\mu>1/2$, we have
\begin{equation*}
k^*= \mathcal{O} \left( \delta^{-\frac{2}{2\mu+3}} \right), \quad
\| f^\delta(k^*) - f^\dagger \| = o\left( \delta^{\frac{2\mu+1}{2\mu+3}} \right) \textrm{~as~} \delta\to 0.
\end{equation*}
\end{remark}

\begin{remark}
It is not difficult to show that Theorem \ref{ThmDis} also holds for the scheme (\ref{symplectic}) with $\Delta t$ replaced by $\Delta t_k$ such that $(\Delta t_k)_k \subset [\Delta t_{min}, \sqrt{2}/\|K\|)$, where $\Delta t_{min}>0$ is a constant.
\end{remark}

\begin{remark}
It should be noted that the combination of the Nesterov's method and the non-symplectic schemes for our second order asymptotical regularization (\ref{SecondFlow}) may also provide an accelerated iterative regularization method. For example, consider the following scheme (it is not a symplectic method as it belongs to explicit numerical scheme)
\begin{equation}\label{symplectic2}
\left\{\begin{array}{l}
q^{k+\frac{1}{2}} = q^{k} - \frac{\Delta t}{2}  \frac{1+2s}{t_{k}} q^{k} + \frac{\Delta t}{2} K^*(y^\delta-K f^{k}), \\
f^{k+1} = f^{k} + \Delta t q^{k+\frac{1}{2}} , \\
v^{k+1} = f^{k+1} + \Delta t b_{k+1} q^{k+\frac{1}{2}} , \\
q^{k+1} = q^{k+\frac{1}{2}} - \frac{\Delta t}{2}\frac{1+2s}{t_{k+1}} q^{k+\frac{1}{2}} + \frac{\Delta t}{2} K^*(y^\delta-K v^{k+1}), \\
f^{0}=f_0, q^{0}=0,
\end{array}\right.
\end{equation}
where
\begin{equation}\label{bk}
b_{k}=\left( 2- \frac{\Delta t(1+2s)}{2 t_{k}}  \right) \left( 1- \frac{\Delta t(1+2s)}{2 t_{k}} \right).
\end{equation}
The above scheme can also be written in the form of (\ref{recurrenceOur}), but with parameters
\begin{equation*}
a_k = \left(1- \frac{\Delta t(1+2s)}{2 t_{k}}  \right)^2, \quad \omega_k=\frac{\Delta t^2}{2} \left( 2 -  \frac{\Delta t(1+2s)}{2 t_{k}} \right).
\end{equation*}
It is not difficult to show that Proposition \ref{BiasFunIneq}, and hence Theorem \ref{ThmDis}, also holds for the scheme (\ref{symplectic2}). Consequently, iteration (\ref{symplectic2}) also offers an accelerated iterative regularization method.
\end{remark}


\section{Application to the diffusion-based bioluminescence tomography (BLT)}\label{sec:bioluminescence}

\subsection{Background of BLT and a reduced mathematical model}
In the modern world, biomedical imaging has become extremely important not only for patient care but also for the study of biological structure and function, and for addressing fundamental questions in biomedicine. In molecular imaging, small animal organs and tissues are often labeled with reporter probes that generate detectable signals that can be tracked outside a living body. This technology has been widely used in clinical medicine for investigating tumorigenesis, cancer metastasis, cardiac diseases, etc. In comparison with traditional biomedical imaging approaches such as X-ray computed tomography, positron emission tomography and ultrasound and magnetic resonance imaging, optical molecular imaging has attracted considerable attention for its cost-effectiveness and performance as it directly reveals molecular and cellular activities sensitively \cite{Du2006}. Among various optical molecular imaging techniques, fluoresence molecular imaging \cite{Ntziachristos2002} and bioluminescence imaging (BLI) \cite{Rice2001} are among the most widely used in practice. In contrast with fluorescence imaging, there is no inherent tissue autofluorescence generated by external illumination in bioluminescence imaging, which makes it extremely sensitive.
However, BLI is primarily qualitative and cannot provide information about the distribution of an \emph{in vivo} bioluminescent source. For the problem of reconstructing an internal bioluminescent source from the measured bioluminescent signal on the external surface of a small animal, a
quantitative prototype, termed as bioluminescence tomography (BLT), is introduced \cite{Han2006}.

Bioluminescent photon propagation in biological tissue is governed by the radiative transfer equation (RTE) which has been utilized as the forward model for bioluminescence tomography \cite{Natterer2001}. However, the RTE is highly dimensional and presents a serious challenge for its accurate numerical simulations given the current level of development in computer software and hardware. Because the mean-free path of the photon is between 500\,nm and 1000\,nm in biological tissues, which is very small compared to the size of a typical object in this context, the predominant phenomenon in BLT is scattering, which provides a diffusion approximation of the RTE by the following reduced mathematical model \cite{Han2006}
\begin{eqnarray}\label{Intro:eq1}
\left\{\begin{array}{ll}
-{\rm div}(D\nabla u)+\mu_a u  =f \chi_{\Omega_0} \quad {\rm in\ }\Omega,  \\
u + 2A D\partial_\nu u = g^- \quad {\rm on}\ \Gamma,
\end{array}\right.
\end{eqnarray}
where $u$ denotes the (direction-averaged) photon density, $D=[3(\mu_a+\mu'_s)]^{-1}$ with $\mu_a$ and $\mu'_s$ being the absorption and reduced scattering coefficients. The boundary $\Gamma$ of the domain $\Omega\subset \mathbb{R}^n$ ($n=2,3$) is assumed to be Lipschitz continuous. $\partial_\nu$ is the outward normal differentiation operator. $\Omega_0\subset\Omega$ is known as a permissible region of the source function, and $\chi$ is the indicator
function such that $\chi_{\Omega_0}(x)=1$ for $x\in \Omega_0$, while
$\chi_{\Omega_0}(x)=0$, when $x\not\in \Omega_0$. $g^-$ is an incoming flux on G and it vanishes when the imaging is implemented in a dark environment. $A= \frac{1+R(x)}{1-R(x)}$  with
$R(x)\approx -1.4399\,\gamma(x)^{-2}+0.7099\,\gamma(x)^{-1}+0.6681+0.0636\,\gamma(x)$ and
$\gamma(x)$ being the refractive index of the medium at $x\in \Gamma$. In the case when $\Omega$ is a unit circle centered at the origin, $\mu_a =0.04, \mu'_s=1.5$ , and $A=3.2$ with refractive index $\gamma=1.3924$. In BLT, the measurement is the outgoing flux density on the boundary:
\begin{equation}
g = -D\partial_\nu u ~ {\rm on\ }\Gamma. \label{Intro:eq3}
\end{equation}
If we denote by $g_1:= g^-+2A\, g$ and $g_2:= -g$, then the BLT problem (\ref{Intro:eq1})-(\ref{Intro:eq3}) can be formulated as Problem \ref{prob:2.1}, i.e., the problem (\ref{eq1})-(\ref{eq3}). This inverse source problem has been intensively studied in \cite{Cheng:2014,Gong2010,Gong2014,Gong2016,Han2006,Han2011,Song:2012,ZhangYe2018} and referenced therein.
The essential methodology in these studies is to solve the inverse problem by a two-step strategy. The first step is to adopt Tikihonov variational regularization with \emph{a priori} regularization parameter choice rule to overcome the ill-poseness of original inverse problem, and then solve the regularized PDE-controlled optimization problem by a numerical algorithm (usually we adopt an iterative method). The defects of these existing methods are: (a) Tikihonov regularization exhibits a ``strong'' saturation phenomenon, i.e., the optimal convergence rate is limited by $\mathcal{O}(\delta^{2/3})$ with respect to H\"older-type source condition and noise level $\delta$ of data. (b) The \emph{a priori} stopping rule of regularization parameter is not realistic in practice as it requires some knowledge of the unknown exact solution. (c) Especially for large-scale inverse problems, variational regularization methods are time consuming. In order to overcome these shortcomings, we shall apply the developed accelerated iterative regularization method (\ref{symplectic}) for the fast solution of Problem \ref{prob:2.1}. It should also be noted that, recently, by assuming the sourcewise representation of source function $f^\dagger$, the authors in \cite{ZhangJIIP2018} combined the coupled complex boundary method and the expanding compacts method to propose a new regularization method that can calculate a posteriori error estimate efficiently. However, no convergence rate can be derived for such a method.

\subsection{Analysis of a mathematical formulation}\label{sec:Analysisformulation}

The aim of this subsection is to reformulate the inverse source problem (\ref{eq1})-(\ref{eq3}) as an abstract operator equation (in a relaxed weak form) so that we can adopt the developed accelerated iterative regularization method with the a posteriori stopping rule in the previous section. We start with the basic assumptions on the system parameter.
\begin{assumption}\label{Assumption1}
$D\in L^\infty(\Omega)$ and $D(x)\geq D_0$ for almost every $x\in \Omega$; $\mu_a\in L^2(\Omega)$ and $\mu_a(x)\geq \mu_0$ for almost every $x\in \Omega$. Here, $D_0$ and $\mu_0$ are two positive constants. Moreover, $\Gamma\subset \mathbb{R}^{n-1}$ is a open bounded set.
\end{assumption}

Denote by
\begin{equation}
\label{Space}
V = \left\{ u: \|u\|_{V} < +\infty \right\}, \quad V_0 = \left\{ u\in V: u=0 {\rm ~a.e. ~ on~} \Gamma \right\},
\end{equation}
where
\begin{equation}
\label{NormV}
\|u\|_{V}= \sqrt{\langle u , u \rangle_{V}}, \quad
\langle u , v \rangle_{V} := \langle \mu_{a}\,u , \,v \rangle_{L^2(\Omega)} + \langle D\nabla u\, , \nabla v \rangle_{L^2(\Omega)},
\end{equation}
is the weighted $H^1(\Omega)$ norm. Moreover, we introduce the norm of the trace space $V^{1/2}(\Gamma)$ by
\begin{equation}
\label{NormTrace}
\|v\|_{V^{1/2}(\Gamma)} := \inf_{u\in V} \left\{ \|u\|_{V}: \gamma_0 u = v \right\},
\end{equation}
where $\gamma_0: V\to V^{1/2}(\Gamma)$ denotes the standard trace operator.
The space $ V^{-1/2}(\Gamma)$ is defined as dual of $V^{1/2}(\Gamma)$, with the norm given
\begin{equation}
\label{NormTrace1}
 \|v\|_{V^{-1/2},\Gamma} := \sup_{w\in V^{1/2}(\Gamma), w\neq 0}\frac{\langle v,w\rangle_{V^{-1/2}(\Gamma),V^{1/2}(\Gamma)}}{\|w\|_{V^{1/2},\Gamma}}.
\end{equation}
It is not difficult to show that all of $V$, $V_0$, $V^{1/2}(\Gamma)$ and $V^{-1/2}(\Gamma)$ are Banach spaces, equipped with the corresponding norms (\ref{NormV}), (\ref{NormTrace}) and (\ref{NormTrace1}), respectively. We remark that if $D=\mu_{a}\equiv1$, $V$, $V^{1/2}(\Gamma)$ and $ V^{-1/2}(\Gamma)$ are reduced to the standard Sobolev spaces $H^1(\Omega)$, $H^{1/2}(\Gamma)$ and $H^{-1/2}(\Gamma)$, respectively. For simplicity, denote $Q_0 = L^2(\Omega_0)$, $Q = L^2(\Omega)$, and $Q_{\Gamma} = L^2(\Gamma)$. Set $V_{g_1}:= \{ v\in V: v=g_1 \, {\rm on}\ \Gamma\}$. Define
\begin{equation} a(u,v) = \int_\Omega \left (D\nabla
u\cdot\nabla v+\mu_{a}\,u\,v\right)\,dx , \quad\forall\,u,v\in
V.\label{bilineara}
\end{equation}
Then $a(\cdot,\cdot)$ is symmetric, continuous and coercive on
$V$. Therefore, by the Lax-Milgram Lemma (\cite{Evans1998}),
for any $f\in Q_0$, the problems
\begin{equation}\label{FinitePro_uD}
u_D(f,g_1)\in V_{g_1},\quad a(u_D(f,g_1),v) =\langle f,v \rangle_{Q_0}, \quad\forall\,v\in V_0\label{weakd}
\end{equation}
and
\begin{equation}\label{FinitePro_uN}
u_N(f,g_2)\in V,\quad a(u_N(f,g_2),v) =\langle f,v \rangle_{Q_0} +\langle g_2,v\rangle_{Q_\Gamma}, \quad\forall\,v\in V
\label{weakn}
\end{equation}
each have a unique solution.
Moreover, a constant $c>0$ exists such that
\begin{eqnarray}
& \|u_D(f,g_1)\|_V\leq c\,(\|f\|_{Q_0}+
\|g_1\|_{V^{1/2}(\Gamma)}),\label{prioriestimated1}\\
& \|u_N(f,g_2)\|_V\leq c\,(\|f\|_{Q_0}+
\|g_2\|_{Q_\Gamma}).\label{prioriestimaten1}
\end{eqnarray}

If we define
\begin{equation*}\label{Def_uD_uN}
\left\{\begin{array}{ll}
u_D(f)=u_D(f,0), \quad u_N(f)=u_N(f,0),  \\
\widetilde{u}_D(g_1)=u_D(0,g_1), \quad \widetilde{u}_N(g_2)=u_N(0,g_2),
\end{array}\right.
\end{equation*}
we obtain that
$u_D(f,g_1) = u_D(f)+\widetilde{u}_D(g_1)$ and
$u_N(f,g_2)= u_N(f)+\widetilde{u}_N(g_2)$.

Define two operators $K_D$ and $K_N$ from $Q_0$ to
$V$ by
\[  K_D\,f=u_D(f),\quad K_N\,q = u_N(f)\quad\forall\,f\in Q_0. \]
Furthermore, define
\begin{equation}\label{Representation}
K := K_D - K_N, \quad y := \widetilde{u}_N(g_2)-\widetilde{u}_D(g_1)\in V.
\end{equation}
It is easy to verify that for any $f\in Q_0$,
\[  K\,f-y = (K_D-K_N)\,f-y
    =u_D(f,g_1)-u_N(f,g_2).  \]
Therefore, $K\,f=y $ means that $u_D(f,g_1)=u_N(f,g_2)$. In other words, the original BLT problem is equivalent to the following problem (in the sense of weak form):  find $f\in Q_0$ such that
\begin{equation}\label{leastH1}
u_D(f,g_1)=u_N(f,g_2)\quad\textrm{~in~}V.
\end{equation}

\begin{proposition}\label{ThmCompact}
The operator $K:Q_0\to V$ is compact.
\end{proposition}

The proof of Proposition \ref{ThmCompact} can be found in Appendix C. Now, let us consider the case with inexact measurement. Suppose that instead of exact boundary data $\{g_1, g_2\}$, we are given noisy data $\{g^\delta_1, g^\delta_2\}$ satisfying the following assumption.
\begin{assumption}\label{Assumption2}
Let $g_1, g^\delta_1\in V^{1/2}(\Gamma)$ and $g_2, g^\delta_2\in V^{-1/2}(\Gamma)$ such that
\begin{equation}\label{NoisyData}
\|g^\delta_1-g_1\|_{V^{1/2}(\Gamma)} + \|g^\delta_2-g_2\|_{V^{-1/2}(\Gamma)} \leq \delta,
\end{equation}
where the noise level $\delta>0$ is known.
\end{assumption}

\begin{remark}
In Assumption \ref{Assumption2} the data space is assumed to be the trace spaces $V^{1/2}(\Gamma)\times V^{-1/2}(\Gamma)$, which is designed for noise-free boundary data of BLT problem. Consequently, the noise in Assumption \ref{Assumption2} is not entirely random, and it also belongs to the considered data space $V^{1/2}(\Gamma)\times V^{-1/2}(\Gamma)$. However, in practice, the original measured data may hardly approximate true Dirichlet data in $V^{1/2}(\Gamma)$ as the noise usually exhibits a weaker regularity, e.g. $g^\delta_1$ usually only belongs to $L^{2}(\Gamma)$. In this case, one can use a smoothing technique to obtain a valid mollification $g^{\delta,\epsilon}_1$ of noisy Dirichlet data  with small enough $\epsilon=\epsilon(\delta)>0$ such that
\begin{equation*}\label{NewNoisyData}
\|g^{\delta,\epsilon}_1-g^\delta_1\|_{L^{2}(\Gamma)} \leq \delta.
\end{equation*}

In this paper, we use the mollification $g^{\delta,\epsilon}_1$, defined through the convolution smoother, i.e.
\begin{equation}\label{mollification}
g^{\delta,\epsilon}_1 = g^\delta_1 \ast \eta_\epsilon := \int_{x'\in\Gamma} g^\delta_1(x') \eta_\epsilon(x-x') d x',
\end{equation}
where the mollifier $\eta_\epsilon$ is defined by
\begin{equation}\label{mollifier}
\eta_\epsilon (x)= \frac{1}{\epsilon^{n-1}} \eta_\epsilon \left( \frac{x}{\epsilon} \right), \quad \eta (x) = \left\{\begin{array}{ll}
C e^{1/(|x|^2-1)}, \quad {\rm if\ }|x|<1,  \\
0, \quad {\rm if\ }|x|\geq1,
\end{array}\right.
\end{equation}
with the positive constant $C$ chosen such that $\int_{\mathbb{R}^{n-1}} \eta (x) dx =1$. According to \cite[Theorem 2.29]{Adams2003}, if ${\rm supp}(g^\delta_1) \Subset \Gamma$ and ${\rm dist}({\rm supp}(g^\delta_1), \partial \Gamma) > \epsilon$, we have $g^{\delta,\epsilon}_1\in C^\infty_0(\Gamma) \subset V^{1/2}(\Gamma)$. It should be noted that when $\partial \Gamma=\emptyset$, e.g. $\Gamma$ is a circle or a sphere, ${\rm dist}({\rm supp}(g^\delta_1), \partial \Gamma)=+\infty$. Moreover, in the case $g^\delta_1\in L^{2}(\Gamma)$, $\|g^{\delta,\epsilon}_1 - g^\delta_1\|_{L^{2}(\Gamma)} \to 0$ as $\epsilon\to0$.
\end{remark}

\begin{proposition}\label{TransformErrorLevel}
Under Assumption \ref{Assumption2}, it holds $\|y^\delta-y\|_V \leq \delta$, where $y^\delta = \widetilde{u}_N(g^\delta_2)-\widetilde{u}_D(g^\delta_1)$.
\end{proposition}

\begin{proof}
Define $v_D:= \widetilde{u}_D(g^\delta_1)- \widetilde{u}_D(g_1)$ and
$v_N:= \widetilde{u}_N(g^\delta_2)- \widetilde{u}_N(g_2)$.
Then, $v_D$ and $v_N$ satisfy the following BVPs
\begin{equation}\label{ProfEq1}
\left\{\begin{array}{ll}
-{\rm div}(D\nabla v_D)+\mu_a v_D = 0 & \textrm{~in~} \Omega, \\
v_D = g^\delta_1-g_1 & \textrm{~on~} \Gamma.
\end{array}\right.
\end{equation}
and
\begin{equation}\label{ProfEq2}
\left\{\begin{array}{ll}
-{\rm div}(D\nabla v_N)+\mu_a v_N = 0 & \textrm{~in~} \Omega, \\
D \frac{\partial v_N}{\partial \mathbf{n}} = g^\delta_2-g_2 & \textrm{~on~} \Gamma.
\end{array}\right.
\end{equation}
Now, let us show that
\begin{equation}\label{embeddingIndentity}
\|v_D\|_{V} = \|g^\delta_1-g_1\|_{V^{1/2}(\Gamma)} \textrm{~and~} \|v_N\|_{V} = \|g^\delta_2-g_2\|_{V^{-1/2}(\Gamma)}.
\end{equation}

By the definition (\ref{NormTrace}), we have
\begin{equation}\label{ProfEq1a}
\|g^\delta_2-g_2\|_{V^{1/2}(\Gamma)}\leq \|v_D\|_{V}.
\end{equation}
On the other hand, according to equation (\ref{ProfEq1}), we have together with (\ref{NormV}) that for any $v\in V$: $\langle v_D , v \rangle_{V}= \int_\Gamma D \frac{\partial v_D}{\partial \mathbf{n}} v ds$, which implies that
\begin{equation*}
\|v_D\|^2_{V} = \left| \int_\Gamma D \frac{\partial v_D}{\partial \mathbf{n}} (g^\delta_1-g_1) ds \right| \leq \|g^\delta_1-g_1\|_{V^{1/2}(\Gamma)} \left\| D \frac{\partial v_D}{\partial \mathbf{n}}\right\|_{V^{-1/2}(\Gamma)}.
\end{equation*}
By the trace theorem, we have $\left\|D \frac{\partial v_D}{\partial \mathbf{n}}\right\|_{V^{-1/2}(\Gamma)}\leq \|v_D\|_{V}$. Hence, we derive
\begin{equation}\label{ProfEq1b}
\|v_D\|_{V} \leq \|g^\delta_1-g_1\|_{V^{1/2}(\Gamma)}.
\end{equation}
Combine (\ref{ProfEq1a}) and (\ref{ProfEq1b}) to obtain the first identity in (\ref{embeddingIndentity}).

Now, consider the second identity in (\ref{embeddingIndentity}). According to (\ref{ProfEq2}), for all $v\in V$, we have together with (\ref{NormV}) that
\begin{equation}\label{ProfEq2a}
\langle v_N , v \rangle_{V,\Omega} := \langle \mu_{a}\,v_N , \,v \rangle_{L^2(\Omega)} + \langle D\nabla v_N\, , \nabla v \rangle_{L^2(\Omega)} = \int_\Gamma (g^\delta_2-g_2) \gamma_0 v ds.
\end{equation}

Set $v=v_N$ to get $\|v_N\|^2_{V} = \left| \int_\Gamma (g^\delta_2-g_2) \gamma_0 v_N ds \right|$, which gives
\begin{eqnarray*}
& \|g^\delta_2-g_2\|_{V^{-1/2}(\Gamma)} = \sup_{\phi\in V^{1/2}(\Gamma)} \frac{\left| \int_\Gamma (g^\delta_2-g_2) \phi ds \right|}{\|\phi\|_{V^{1/2}(\Gamma)}} \\ & \qquad
\geq_{\phi=\gamma_0 v_N} \frac{\left| \int_\Gamma (g^\delta_2-g_2) \gamma_0 v_N ds \right|}{\|\gamma_0 v_N\|_{V^{1/2}(\Gamma)}} = \frac{\|v_N\|^2_{V}}{\|\gamma_0 v_N\|_{V^{1/2}(\Gamma)}}.
\end{eqnarray*}
The above inequality together with the trace inequality, i.e. $\|\gamma_0 v_N\|_{V^{1/2}(\Gamma)}\leq \|v_N\|_{V}$, gives
\begin{equation}\label{ProfIneq2a}
\|v_N\|_{V} \leq \|g^\delta_1-g_1\|_{V^{-1/2}(\Gamma)}.
\end{equation}

On the other hand, by using (\ref{ProfEq2a}) we deduce that
\begin{eqnarray*}
& \|g^\delta_2-g_2\|_{V^{-1/2}(\Gamma)} = \sup_{\phi\in V^{1/2}(\Gamma)} \frac{\left| \int_\Gamma (g^\delta_2-g_2) \phi ds \right|}{\|\phi\|_{V^{1/2}(\Gamma)}} = \sup_{\phi\in V^{1/2}(\Gamma)} \frac{\left| \langle v_N , \gamma^{-1}_0 \phi \rangle_{V,\Omega}\right|} {\|\phi\|_{V^{1/2}(\Gamma)}}
\\ & \qquad
\leq \|v_N\|_{V} \cdot \sup_{\phi\in V^{1/2}(\Gamma)} \frac{\|\gamma^{-1}_0 \phi\|_{V}} {\|\phi\|_{V^{1/2}(\Gamma)}} \leq \|v_N\|_{V},
\end{eqnarray*}
which implies the second identity of (\ref{embeddingIndentity}) by noting (\ref{ProfIneq2a}).

Finally, by using the definition of $y^\delta$ and identities (\ref{embeddingIndentity}), we complete the proof by following inequalities
\begin{equation*}
\|y^\delta-y\|_V = \|v_N-v_D\|_{V} \leq \|g^\delta_1-g_1\|_{V^{1/2}(\Gamma)} + \|g^\delta_2-g_2\|_{V^{-1/2}(\Gamma)} \leq \delta.
\end{equation*}
\end{proof}

Next we discuss the form of $K^*\,(K\,f-y^\delta)$, which is used in our main algorithms (\ref{symplectic}) and (\ref{symplectic2}), in the context of the BLT problem. To this end, denote by $K^*_D$ and $K^*_N$ the adjoint
operators of $K_D$ and $K_N$ such that for any $v\in V$ and $f\in Q_0$:
\begin{equation*}
\langle K^*_D\,v,f \rangle_{Q_0} = \langle v,K_D\,f \rangle_{V},\quad \langle K^*_N\,v,f \rangle_{Q_0} = \langle v,K_N\,f\rangle_{V}.
\end{equation*}
Then $K^*: V\rightarrow Q_0$ is such that $K^* = K^*_D-K^*_N$.

For any $f\in Q_0$, denote by $u_{DN}(f)=K\,f-y^\delta=u_D(f,g^\delta_1)-u_N(f,g^\delta_2)$, and
define $w_D=w_D(u_{DN}(f))\in V_0$ and $w_N=w_N(u_{DN}(f))\in V$ the
solutions of the adjoint variational problems
\begin{equation}
a(v,w_D) =\langle u_{DN},v \rangle_{V},\quad\forall\,v\in
V_0\label{adjointd}
\end{equation}
and
\begin{equation}
a(v,w_N) =\langle u_{DN},v \rangle_{V},\quad\forall\,v\in
V,\label{adjointn}
\end{equation}
respectively. Then $K^*_D(K\,f-y^\delta)=w_D(u_{DN}(f))|_{\Omega_0}$
and $K^*_N(K\,f-y^\delta)=w_N(u_{DN}(f))|_{\Omega_0}$. Thus, we have
\begin{equation}\label{gradient}
K^*\,(K\,f-y^\delta) = (K^*_D-K^*_N)(K\,f-y^\delta) = [w_D(u_{DN}(f))-w_N(u_{DN}(f))]|_{\Omega_0}.
\end{equation}

Similarly, we can give a form of  the source condition (\ref{SourceCondition}) for our BLT problem.
In fact, (\ref{SourceCondition}) with $\mu=1$ reads that there exists an element $v_*\in Q_0$ such that
\begin{equation}\label{SourceConditionBLT}
f_0 - f^\dagger = (w_D^*-w_N^*)\chi_{\Omega_0},
\end{equation}
where $w_D^*$ and $w_N^*$ are the solutions of (\ref{adjointd}) and (\ref{adjointn}), both with $u_{DN}$ being replaced by $u_D(v_0)-u_N(v_0)$, and $u_D(v_0)=u_D(v_0,0), u_N(v_0)=u_N(v_0,0)$.
%


\section{Numerical experiments}\label{sec:Simulations}

In this section, we devote ourselves to presenting some numerical examples for demonstrating the effectiveness of the proposed accelerated iterative regularization method (\ref{symplectic}).
We take the diffusion-based bioluminescence tomography considered in Section \ref{sec:bioluminescence} as example.
To that end, with the problem domain $\Omega$, parameters $\mu_a, \mu'_s, A$, Robin data $g^-$, and a prescribed
true source function $f^*$, we solve the forward BVP (\ref{Intro:eq1}) to get $u^*$. A finite element method of solving (\ref{Intro:eq1}) is briefly discussed in Appendix D.

The outgoing flux density and the Cauchy data on the boundary are
$$
g=-D\partial_\nu u^*\mid_\Gamma=\frac{1}{2\,A} (u^*-g^-),\quad g_1= g^-+2A\, g,\quad g_2= -g.
$$
Uniformly distributed noises with the relative error level $\delta'$ are added to $g$ to get $g^\delta$
\[ g^{\delta}(x)=[1+\delta'\cdot(2\,\textrm{rand}(x)-1)]\,g(x),\quad x\in \Gamma,\]
where rand$(x)$ returns a pseudo-random value drawn from a uniform distribution on $[0, 1]$. The corresponding noisy Cauchy data
are $g^\delta_1= g^-+2A\, g^\delta$ and $g^\delta_2= -g^\delta$.
Then the noise level of the measurement data is calculated by $\delta=
\|y^{h, \delta}-y^h\|_{V}$, with $y^h = \widetilde{u}^h_N(g_2)-\widetilde{u}^h_D(g_1)$
and $y^{h, \delta} = \widetilde{u}^h_N(g^\delta_2)-\widetilde{u}^h_D(g^\delta_1)$. Here and later on, the superscript $^h$ (or the subscript $_h$) denotes the linear finite element approximation of an element on a consistent triangulation, i.e. $\widetilde{u}^h_N$ and $\widetilde{u}^h_D$ are defined on the same triangulation with maximum triangle diameter $h$, see Appendix D for more details. Without loss of generality, in this section, let $\mu_a=0.04, \mu'_s=1.5, D=1/[3(\mu_a+\mu'_s)]$,
$A=3.2$, and $g^-=0$, which means the imaging is implemented in a dark environment.

Then, with the noisy data $g^{\delta}_1$ and $g^{\delta}_2$, properly chosen parameters, e.g. $s$ and $\Delta t$,
approximate sources $f^k$ are computed by the proposed accelerated iterative regularization method (\ref{symplectic}).
For the BLT problem, (\ref{symplectic}) is reduced to
\begin{equation}\label{symplecticBLT}
\left\{\begin{array}{l}
q^{k+\frac{1}{2}} = q^{k} - \frac{\Delta t}{2}  \frac{1+2s}{t_{k+1}} q^{k+\frac{1}{2}} - \frac{\Delta t}{2} (w^k_D-w^k_N)\chi_{\Omega_0} , \\
f^{k+1} = f^{k} + \Delta t q^{k+\frac{1}{2}} , \\
v^{k+1} = f^{k+1} + 2 \Delta t \frac{2k-2s+3}{2k + 2s+5} q^{k+\frac{1}{2}} , \\
q^{k+1} = q^{k+\frac{1}{2}} - \frac{\Delta t}{2}\frac{1+2s}{t_{k+1}} q^{k+\frac{1}{2}} - \frac{\Delta t}{2}(w^{k+1}_D-w^{k+1}_N)\chi_{\Omega_0} , \\
f^{0}=f_0, q^{0}=0,
\end{array}\right.
\end{equation}
where $w^k_D$ and $w^k_N$ are the solutions of (\ref{adjointd}) and (\ref{adjointn}) respectively,
both with $u_{DN}(f)$ replaced by $u_{DN}(f^k)$. $w^{k+1}_D$ and $w^{k+1}_N$ have similar definitions.
In the following, for the conciseness of the statements, we only consider the case that using Morozov's discrepancy
principle (\ref{discrepancy2}) to control the iterative procedure, namely that the iteration stops when
$\|y^{h,\delta}-Af^{h,k}\|_{V} =\|u^h_D(f^{h,k},g^\delta_1)-u^h_N(f^{h,k},g^\delta_2)\|_{V} \leq \tau \delta$, where
$u^h_D(f^{h,k},g^\delta_1)$ and $u^h_N(f^{h,k},g^\delta_2)$ are the finite element solutions of (\ref{weakd}) and (\ref{weakn}), both with
$f$ being replaced by $f^{h,k}$, and with $g_1$ and $g_2$ being replaced by $g^\delta_1$ and $g^\delta_2$, respectively.
 Moreover, the initial guess $f_0$ is chosen
so that the condition of Lemma \ref{Rootdiscrepancy} is satisfied: $\|y^{h,\delta}-A f^h_0\|_{V} =\|u^h_D(f^h_0,g^\delta_1)-u^h_N(f^h_0,g^\delta_2)\|_{V} >\tau \delta$, where $u^h_D(f_0,g^\delta_1)$ and $u^h_N(f^h_0,g^\delta_2)$ have similar definitions as $u^h_D(f^{h,k},g^\delta_1)$ and $u^h_N(f^{h,k},g^\delta_2)$ above.

We use $N_{\max}:=50000$ as the maximal number of iterations where the iteration (\ref{symplecticBLT}) stops in all of simulations.
To assess the accuracy of the approximate
solutions, we define the finite element approximate $L^2$-norm relative error for an approximate solution $f^{k}$:
L2Err$_k:= \|f^{k}_h-f^*_h\|_{0,\Omega_0} / \|f^*_h\|_{0,\Omega_0} $. Obviously, $\|f^{k}_h-f^*_h\|_{0,\Omega_0} / \|f^*_h\|_{0,\Omega_0} \to \|f^{k}-f^*\|_{0,\Omega_0} / \|f^*\|_{0,\Omega_0}$ as $h\to0$.
All experiments in Subsection \ref{subsec:parameter}--\ref{subsec:comparison} are
implemented for the following two examples:


\textbf{Example 1}: $\Omega:=\{(x_1,x_2)\in\mathbb{R}^2|\,x_1^2+x^2_2<1\}$,
$f^*(x_1,x_2)=(1+x_1+x_2)\chi_{\Omega_0}$ with
$\Omega_0:=\{(x_1,x_2)\in\mathbb{R}^2|\,-0.5<x_1,x_2<0.5\}$.   The measurements are computed on a
mesh with mesh size $h=0.01386$, 144929 nodes and 288768 elements.


\textbf{Example 2}: $\Omega$ is the same as Example 1,
$f^*(x_1,x_2)=(1+x_1+x_2)\chi_{\Omega_1}+e^{1+x_1+x_2}\chi_{\Omega_2}$
with $\Omega_1:=\{(x_1,x_2)\in\mathbb{R}^2|\,(x_1+0.5)^2+x^2_2<0.01\}$
and $\Omega_2:=\{(x_1,x_2)\in\mathbb{R}^2|\,(x_1-0.5)^2+x^2_2<0.01\}$.
The measurements are computed on a mesh with $h=0.01228$, 156225 nodes and 311296 elements.


For Example 1, all approximate sources are
reconstructed over a mesh with mesh size $h=0.0744$, 2325 nodes and 4512 elements.
For Example 2, all approximate sources are
reconstructed over a mesh with mesh size $h=0.0678$, 2505 nodes and 4864 elements.

\subsection{Influence of parameters}\label{subsec:parameter}

The purpose of this subsection is to explore the dependence of the solution accuracy
and the convergence speed on $\tau>0$, time step size $\Delta t$,
model parameter $s$, and
thus to give a guide on the choices for them in practice. For focusing on
the effect of these parameters on the iteration (\ref{symplecticBLT}), we fix $\delta'= 0.1\%$
in this subsection. Moreover, in the remaining part of this section, we simply set $f_0=0, q_0=0$.

We first investigate the influence of parameter $\tau$ on the convergence rate.
For this purpose, we additionally set $s=2$, and $\Delta t=0.06$ for Example 1 or
$\Delta t=0.125$ for Example 2. The detailed iterative numbers
$k^*$ and the corresponding L2-norm relative errors `L2Err$_{k^*}$' for
 different values of $\tau$ are shown in Table \ref{tab:EkVsTau1},
which shows that the smaller $\tau$ is, the more the iterative number
for stopping (\ref{symplecticBLT}) is. It is no surprise because the parameter $\tau$
does not involve the computation of the approximate solutions itself. It is used in the
stop criterion and affects only the iterative number at which the iteration (\ref{symplecticBLT}) stops.
Table \ref{tab:EkVsTau1} indicates using a $\tau<1$ makes the iterative number increase dramatically.
In the remaining experiments, we choose $\tau=1.1$ for Example 1 and $\tau=10$ for Example 2.

\begin{table}[!tbh]
{\footnotesize
\caption{The iterative number $k^*$ and the corresponding relative error $L2Err_{k^*}$ vs $\tau$.}\label{tab:EkVsTau1}
\begin{center}
\begin{tabular}{|c|c|c|c|c|c|c|} \hline
\multirow{2}{*}{$\tau$} &
\multicolumn{2}{c|}{\textbf{Example 1}} &
\multicolumn{2}{c|}{\textbf{Example 2}}  \\
\cline{2-5}
 & L2Err$_{k^*}$ & $k^*$ & L2Err$_{k^*}$ & $k^*$  \\ \hline
$2^{-1}$  & 1.8066e-3 &$N_{\max}$ & 1.6067e-2 & $N_{\max}$\\
$1$       &  4.8745e-3 & 246 & 2.2110e-2 & 4384 \\
$2$       &  5.1612e-3 & 245 & 2.4860e-2 & 3318 \\
$2^{2}$  &  5.5573e-3 & 244 & 3.1812e-2 & 2226 \\
$2^{3}$  &  6.6180e-3 & 242 & 3.9872e-2 & 1255 \\
$2^{4}$  &  8.6874e-3 & 239 & 4.8587e-2 & 966 \\
$2^{5}$  &  1.4807e-2 & 232 & 6.8973e-2 & 560 \\
$2^{6}$  &  2.6672e-2 & 221 & 8.7679e-2 & 362 \\
$2^{7}$  &  5.2232e-2 & 202 & 1.1497e-1 & 129 \\ \hline
\end{tabular}
\end{center}
}
\end{table}

Now we investigate the influence of time step size $\Delta t$ on the
solution accuracy and the convergence rate. To this end, set $s = 2$, and $\tau=1.1$ for Example 1,
$\tau=10$ for Example 2. The iterative numbers $k^*$ and the corresponding
L2-norm relative errors `L2Err$_{k^*}$' are given in Table \ref{tab:EkVsDt1}, which shows that the
bigger the time step size $\Delta t$ is, the faster the iteration is. The iterative number halves
when the $\Delta t$ doubles. However, our experiments suggest that $\Delta t$ should not be too big,
e.g. $\Delta t\leq 0.0625$ in Example 1 and $\Delta t\leq 0.125$ in Example 2.
Otherwise, the iteration will blow up as it breaks the consistency of the numerical scheme.
In the remaining experiments, we choose $\Delta t=0.0625$ for Example 1 and $\Delta t=0.125$ for Example 2.

\begin{table}[H]
{\footnotesize
\caption{The iterative number $k^*$ and the corresponding relative error L2Err$_{k^*}$ vs $\Delta t$.}\label{tab:EkVsDt1}
\begin{center}
\begin{tabular}{|c|c|c|c|c|} \hline
\multirow{2}{*}{$\Delta t$} &
\multicolumn{2}{c|}{\textbf{Example 1}} &
\multicolumn{2}{c|}{\textbf{Example 2}}  \\
\cline{2-5}
 & L2Err$_{k^*}$ & $k^*$ & L2Err$_{k^*}$ & $k^*$  \\ \hline
$2^{-10}$  & 4.6816e-3 & 14964 & 8.4071e-2 & $N_{\max}$\\
$2^{-9}$  & 4.6835e-3 & 7483 & 5.7224e-2 & $N_{\max}$\\
$2^{-8}$  & 4.6852e-3 & 3743 & 4.1839e-2 & 37270 \\
$2^{-7}$  & 4.6891e-3 & 1873 & 4.1849e-2 & 18629 \\
$2^{-6}$  & 4.7268e-3 & 937 & 4.1866e-2 & 9309 \\
$2^{-5}$  & 4.7560e-3 & 470 & 4.1894e-2 & 4650 \\
$2^{-4}$  & 4.9548e-3 & 236 & 4.1945e-2 & 2321 \\
$2^{-3}$  & Divergence & - & 4.1988e-2 & 1159 \\
$2^{-2}$  & Divergence & - & Divergence & -\\ \hline
\end{tabular}
\end{center}
}
\end{table}

\begin{table}[!b]
{\footnotesize
\caption{The iterative number $k^*$ and the corresponding relative error L2Err$_{k^*}$ vs $s$.}\label{tab:EkVsS1}
\begin{center}
\begin{tabular}{|c|c|c|c|c|} \hline
\multirow{2}{*}{$s$} &
\multicolumn{2}{c|}{\textbf{Example 1}} &
\multicolumn{2}{c|}{\textbf{Example 2}}  \\
\cline{2-5}
 & L2Err$_{k^*}$ & $k^*$ & L2Err$_{k^*}$ & $k^*$   \\ \hline
-0.499 & 6.1991e-3 & 2506 & 2.5643e-2 & 1414\\
-0.4 & 9.5068e-3 & 651 & 2.3717e-2 & 4276\\
-0.3 & 1.0187e-2 & 229 & 3.1115e-2 & 3490\\
-0.2 & 6.3593e-3 & 666 & 2.6702e-2 & 1506\\
-0.1 & 7.6404e-3 & 244 & 3.8602e-2 & 703\\
0 & 5.2214e-3 & 681 & 3.8309e-2 & 708\\
$2^{-5}$  & 5.4304e-3  & 540 & 3.8893e-2 & 709 \\
$2^{-4}$  & 6.4303e-3  & 256 & 3.9249e-2 & 710 \\
$2^{-3}$  & 5.1272e-3  & 547 & 4.0054e-2 & 713 \\
$2^{-2}$  & 5.0820e-3  & 413 & 3.9117e-2 & 798 \\
$2^{-1}$  & 5.2929e-3  & 287 & 4.1006e-2 & 808 \\
$1$       & 4.8956e-3  & 321 & 4.2226e-2 & 912 \\
$2$       & 4.9548e-3  & 236 & 4.1988e-2 & 1159 \\
$2^{2}$  & 4.7348e-3  & 349 & 4.1857e-2 & 1548 \\
$2^{3}$  & 4.6814e-3  & 547 & 4.1702e-2 & 2137\\
$2^{4}$  & 4.6760e-3  & 884 & 4.1598e-2 & 2991 \\
$2^{5}$  & 4.6744e-3 & 1316 & 4.1532e-2 & 4210 \\ \hline
\end{tabular}
\end{center}
}
\end{table}

Finally, we discuss the influence of the model parameter $s$ on the
solution accuracy and the convergence rate. In the experiments,
set $\tau=1.1, \Delta t= 0.0625$ for Example 1 and $\tau=10, \Delta t= 0.125$ for Example 2. The required number of iterations $k^*$ and the corresponding relative error L2Err$_{k^*}$ for different values of $s$ are given in Table \ref{tab:EkVsS1}, which indicates that in general,
the value of $s$ is neither small nor big. Specifically, on one hand, small values of $s$, e.g. $s\leq 1$, usually bring the oscillation in solution accuracy during iterations (cf. Figures \ref{fig:EkVsk1}-\ref{fig:EkVsk2}); on the other hand, large values of $s$ require more iterative number but with a little improvement on the solution accuracy. It is suggested that a value of $s$ near 2 produces satisfactory results in both solution accuracy and the iterative number for both Examples 1 and 2, which coincides with the empirical results about the optimal parameter choice for the Nesterov's method. Therefore, in the remaining experiments, we set $s=2$.

\begin{figure}[H]
\subfigure[]{\includegraphics[width=0.31\textwidth]{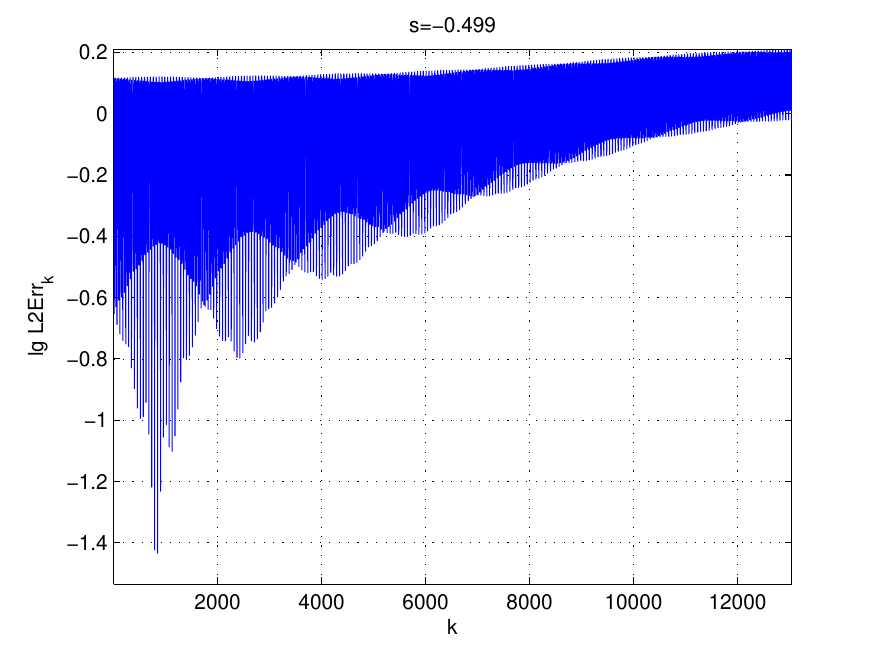}}
\subfigure[]{\includegraphics[width=0.31\textwidth]{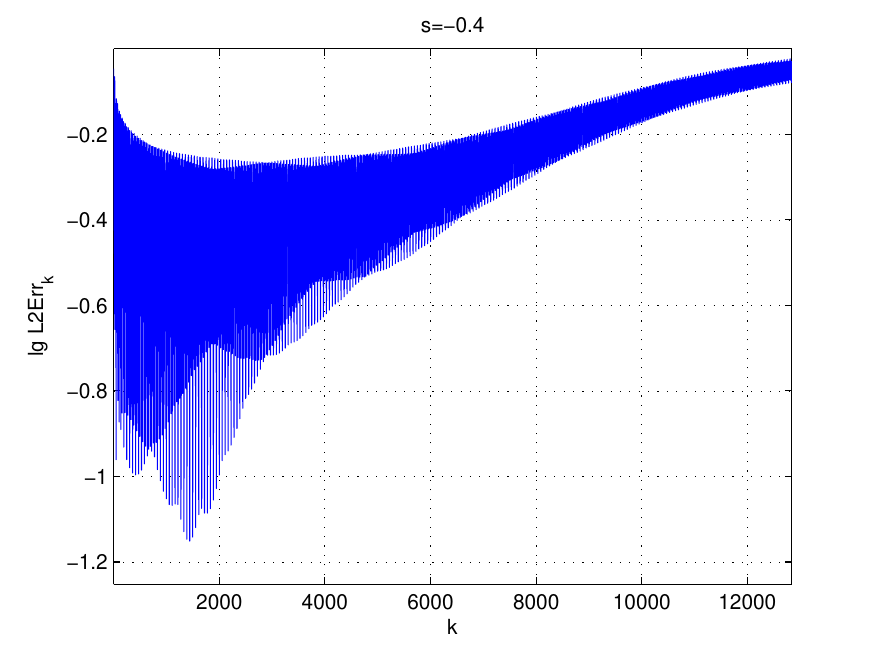}}
\subfigure[]{\includegraphics[width=0.31\textwidth]{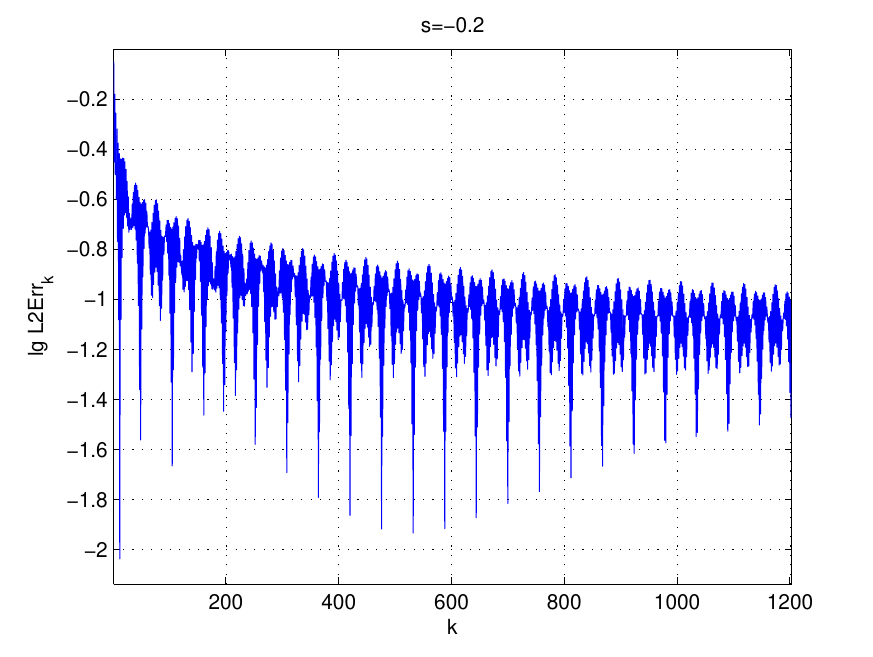}}
\subfigure[]{\includegraphics[width=0.31\textwidth]{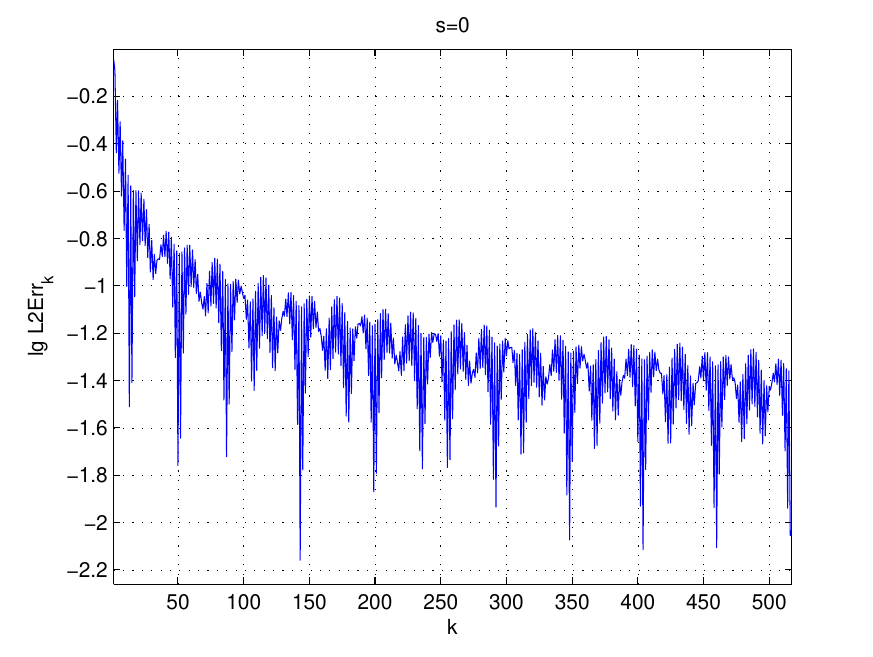}}
\subfigure[]{\includegraphics[width=0.31\textwidth]{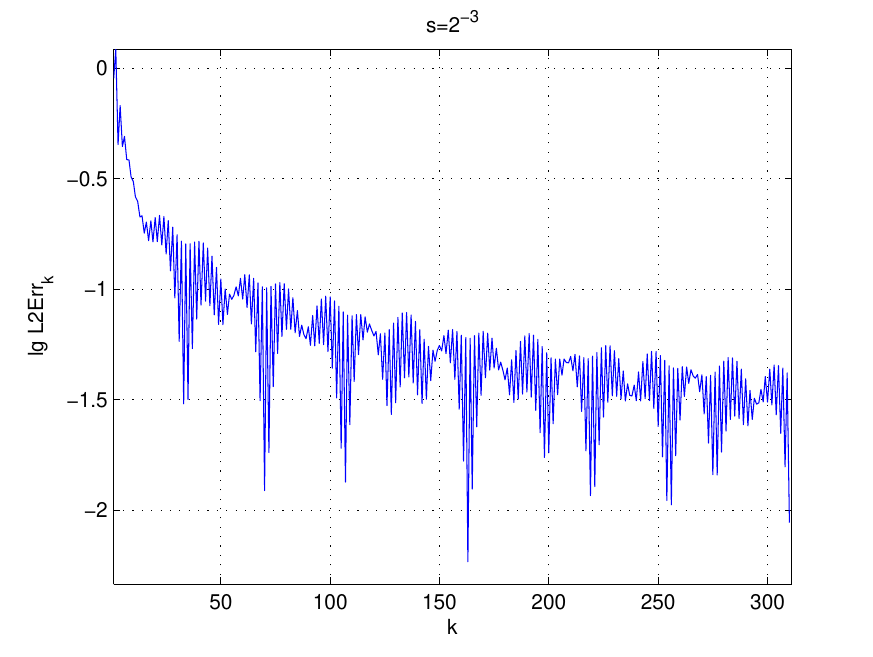}}
\subfigure[]{\includegraphics[width=0.31\textwidth]{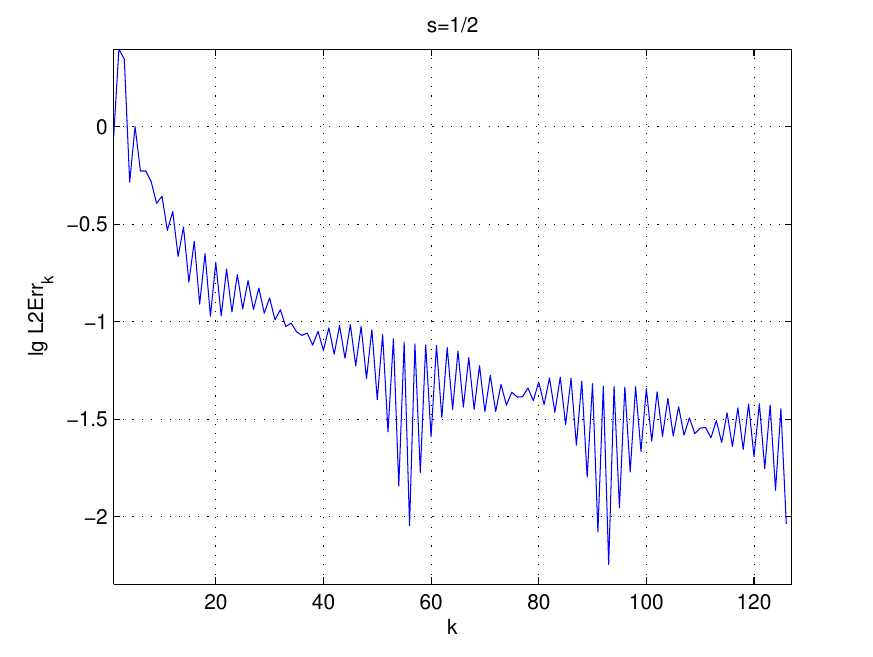}}
\subfigure[]{\includegraphics[width=0.33\textwidth]{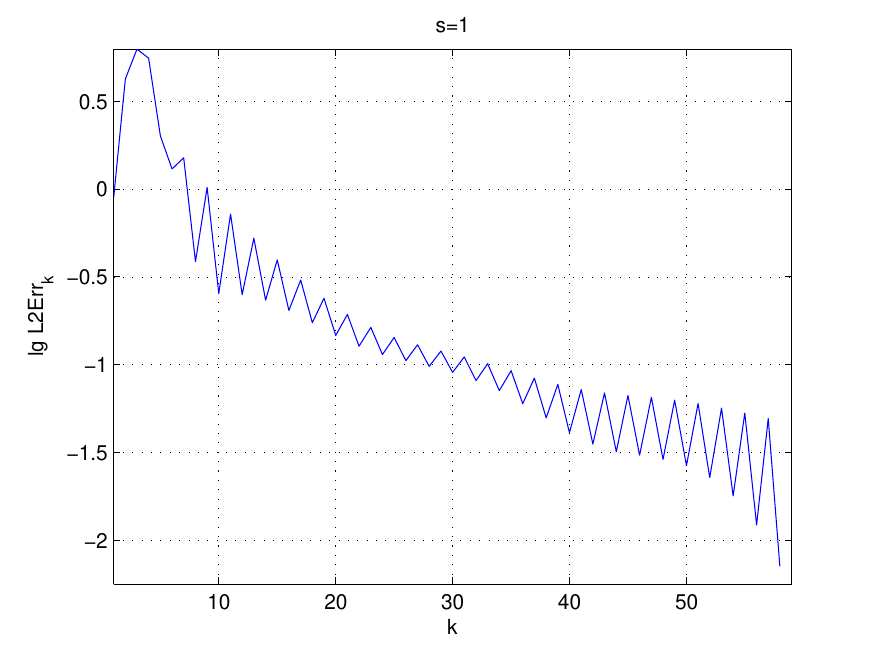}}
\subfigure[]{\includegraphics[width=0.33\textwidth]{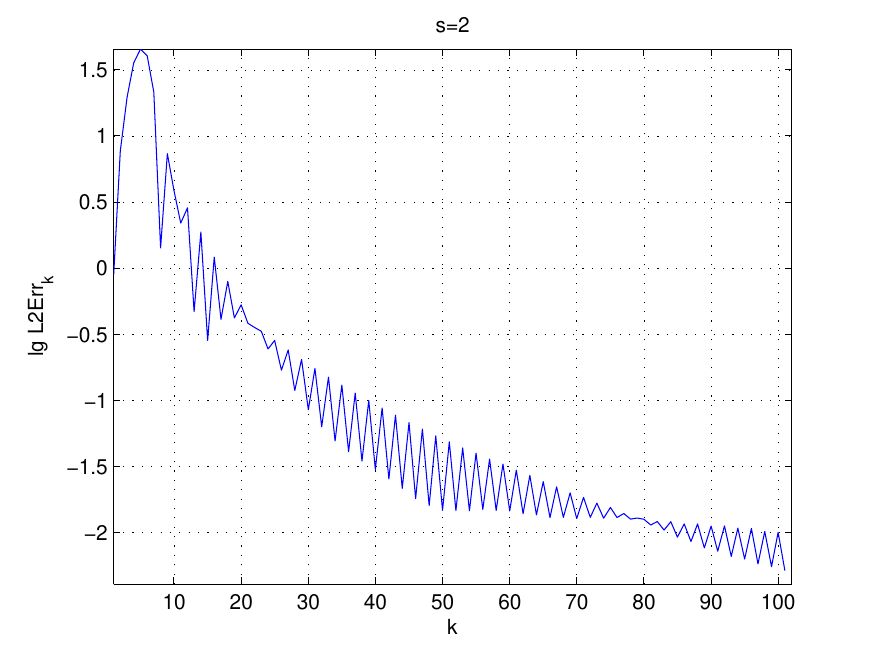}}
\subfigure[]{\includegraphics[width=0.33\textwidth]{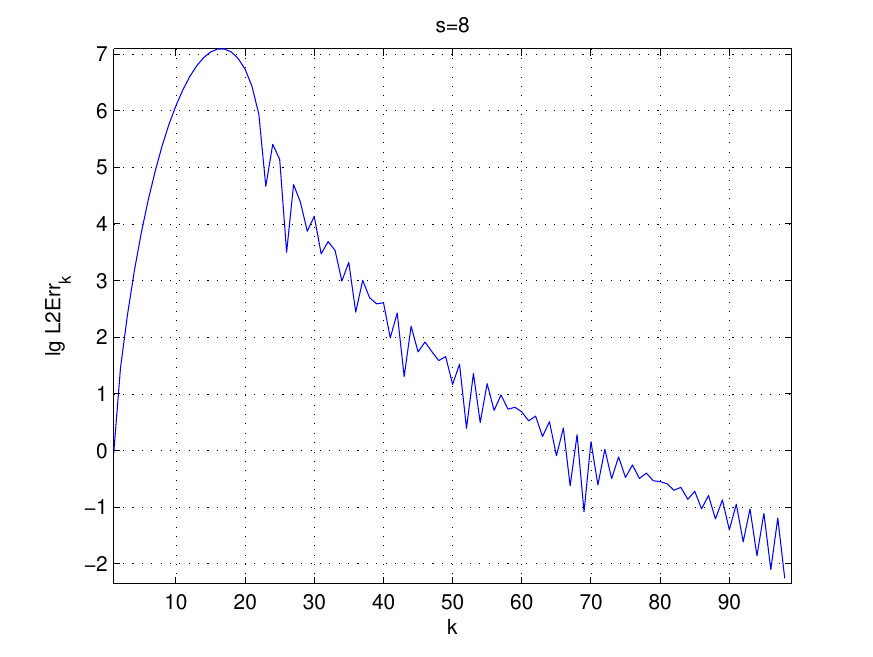}}
\caption{Evolutions of L2Err$_{k^*}$ (on a logarithmic scale) vs. $k$ for different values of $s$ (Example 1). }\label{fig:EkVsk1}
\end{figure}

\begin{figure}[!t]
\subfigure[]{\includegraphics[width=0.31\textwidth]{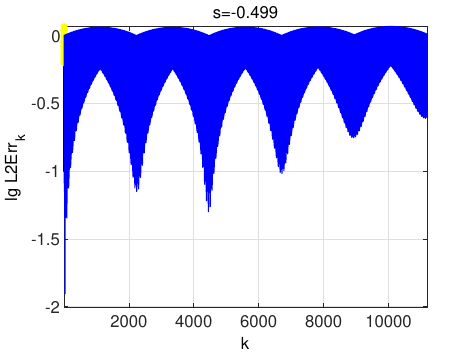}}
\subfigure[]{\includegraphics[width=0.31\textwidth]{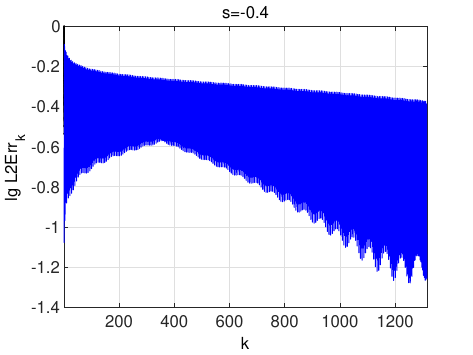}}
\subfigure[]{\includegraphics[width=0.31\textwidth]{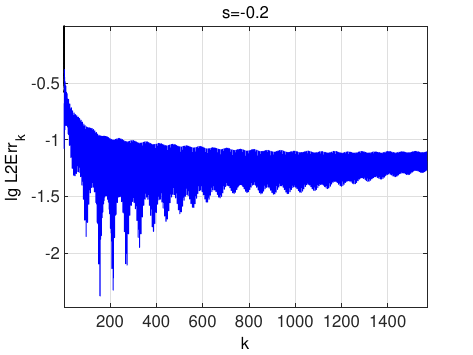}}
\subfigure[]{\includegraphics[width=0.31\textwidth]{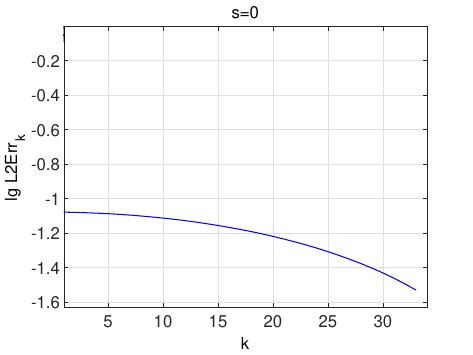}}
\subfigure[]{\includegraphics[width=0.31\textwidth]{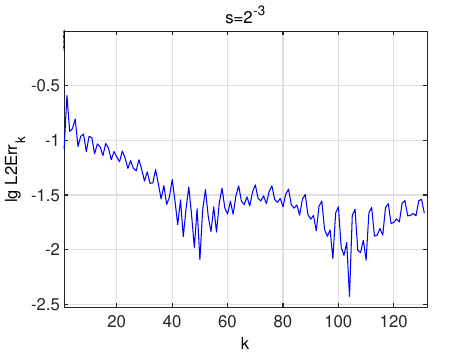}}
\subfigure[]{\includegraphics[width=0.31\textwidth]{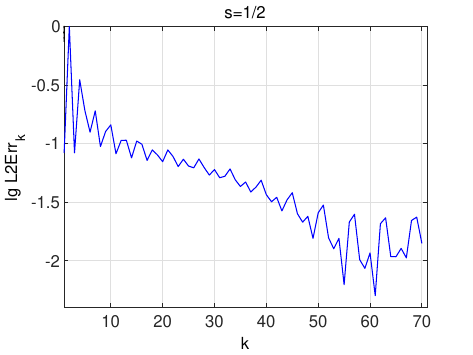}}
\subfigure[]{\includegraphics[width=0.33\textwidth]{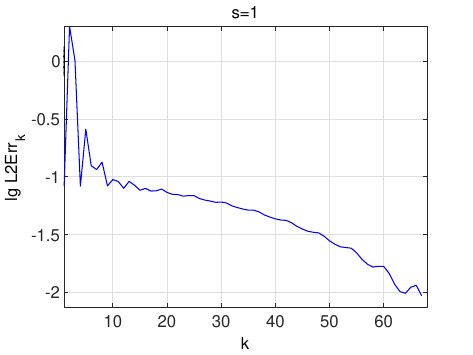}}
\subfigure[]{\includegraphics[width=0.33\textwidth]{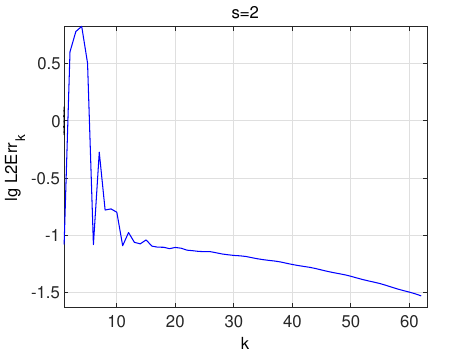}}
\subfigure[]{\includegraphics[width=0.33\textwidth]{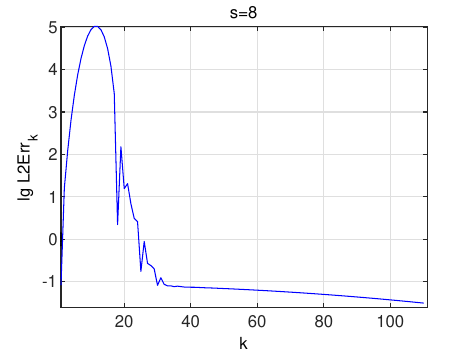}}
\caption{Evolutions of L2Err$_{k^*}$ vs. $k$ for different values of $s$ (Example 2). }\label{fig:EkVsk2}
\end{figure}

\subsection{Comparison with other methods}\label{subsec:comparison}

In this subsection, we compare the behaviors regarding the solution accuracy
and the convergence rate between scheme (\ref{symplecticBLT}),  the non-symplectic scheme
(\ref{symplectic2}), two well-known acceleration methods: the Nesterov's method, the $\nu$-method,
and the Landweber method (\ref{Landweber}).
For the BLT: Problem \ref{prob:2.1}, the non-symplectic (\ref{symplectic2}) has the form
\begin{equation}\label{symplectic2BLT}
\left\{\begin{array}{l}
q^{k+\frac{1}{2}} = q^{k} - \frac{\Delta t}{2}  \frac{1+2s}{t_{k+1}} q^{k}- \frac{\Delta t}{2}(w^{k}_D-w^{k}_N)\chi_{\Omega_0} , \\
f^{k+1} = f^{k} + \Delta t q^{k+\frac{1}{2}} , \\
v^{k+1} = f^{k+1} + \Delta t b_{k+1} q^{k+\frac{1}{2}} , \\
q^{k+1} = q^{k+\frac{1}{2}} - \frac{\Delta t}{2}\frac{1+2s}{t_{k+1}} q^{k+\frac{1}{2}} - \frac{\Delta t}{2}(w^{k+1}_D-w^{k+1}_N)\chi_{\Omega_0} , \\
f^{0}=f_0, q^{0}=0,
\end{array}\right.
\end{equation}
where $b_{k+1}$ is given in (\ref{bk}).

In our simulations, for the BLT problem, the $\nu$-method is defined by
\begin{eqnarray}\label{nuMethod}
\left\{\begin{array}{l}
f^{k} = f^{k-1} + \mu_k(f^{k-1}-f^{k-2})- \omega \cdot \omega_k \cdot (w^{k-1}_D-w^{k-1}_N)\chi_{\Omega_0} , \\
f^0=f^{-1}=f_0
\end{array}
\right.
\end{eqnarray}
with $\mu_1=0, \omega_1=(4\nu+2)/(4\nu+1)$ and
\begin{eqnarray*}
\mu_k= \frac{(k-1)(2k-3)(2k+2\nu-1)}{(k+2\nu-1)(2k+4\nu-1)(2k+2\nu-3)}, \\
\omega_k= 4 \frac{(2k+2\nu-1)(k+\nu-1)}{(k+2\nu-1)(2k+4\nu-1)} \textrm{~for~} k>1,
\end{eqnarray*}
where $\omega>0$ is the weight.
Note that $\omega=1$ in the conversional $\nu$-method (cf. \cite[\S~6.3]{engl1996regularization})
as it works for a normalized operator equation $Kf=y$ with $\|K\|\leq1$. For our BLT problem,
$\omega$ in (\ref{nuMethod}) plays the role of normalization, and it can be set as
$\omega=\omega_{norm}(:=1/\|K^*K\|)$, which can be calculated by
\begin{eqnarray*}
\omega_{norm}& =\frac{\|1\|_{Q_0}}{\|(w_D(1,g^\delta_1)-w_N(1,g^\delta_2))-(w_D(0,g^\delta_1)-w_N(0,g^\delta_2))\|_{Q_0}}.
\end{eqnarray*}

The dependence of stopping iterative number $k^*$
and the corresponding relative error L2Err$_{k^*}$ on $\nu$ for different values of noise level $\delta$ are plotted
under $\log$-scale in Figure \ref{fig:nuDependence}, from which we can see that for both examples, a
moderate value of $\nu$ have the best efficiency in both solution accuracy and iterative number: large values of $\nu$
do not improve the solution accuracy too much while require much more iterative numbers. In both examples, the $\nu$ of the best efficiency is near 1.  In Table \ref{tab:comparison}, we report relative error L2Err$_{k^*}$ as well as the corresponding iterations numbers of $\nu$-method with $\nu=1/4$, $\nu=1/2$, $\nu=1$ and $\nu=2$.

\begin{figure}[H]
\subfigure[]{\includegraphics[width=0.5\textwidth]{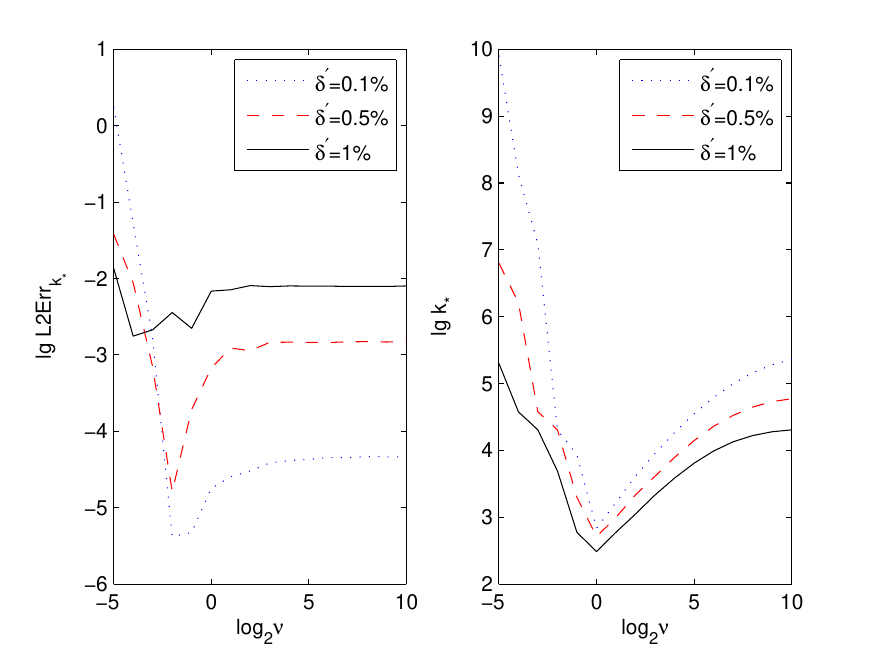}}
\subfigure[]{\includegraphics[width=0.5\textwidth]{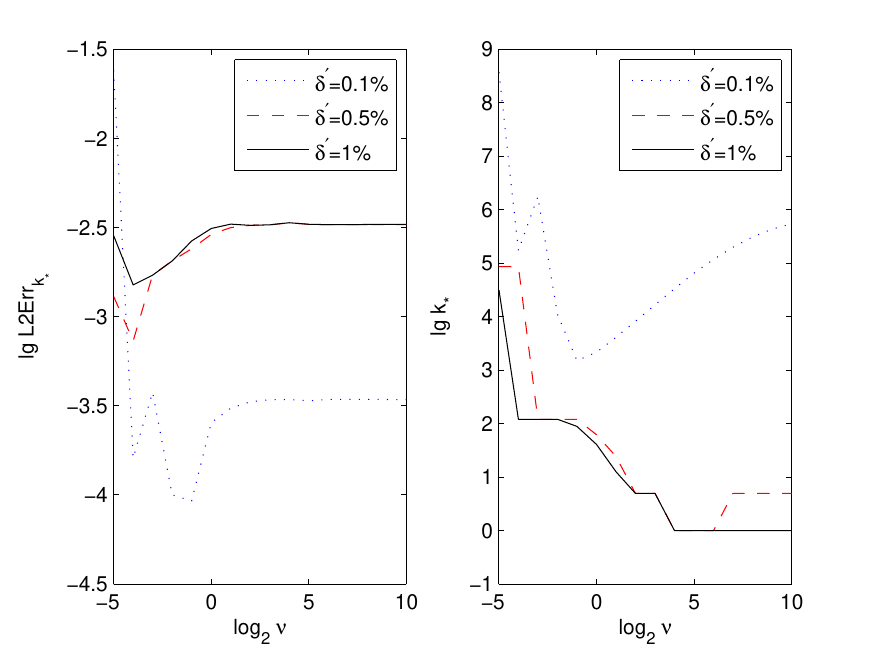}}
\caption{Dependence of iterative number $k^*$ and the corresponding L2Err$_{k^*}$ on $\nu$ for different values of noisy level $\delta$. (a): Example 1; (b): Example 2. }\label{fig:nuDependence}
\end{figure}

The Nesterov's method is defined by~(\cite{Neubauer-2017})
\begin{eqnarray}\label{Nesterov}
\left\{\begin{array}{ll}
z_{k} = f^{k} + \frac{k-1}{k+\alpha-1} \left(f^{k} - f^{k-1} \right),  \\
f^{k+1} = z_{k} - \omega (w^{k}_D-w^{k}_N)\chi_{\Omega_0} ,\\
f^0=f^{-1}=f_0
\end{array}\right.
\end{eqnarray}
with $\alpha\geq 3$ and $0<\omega\leq\omega_{norm}$. In all simulations, we choose $\alpha=3$. For Example 1,
$\omega_{norm}\approx 0.005422264152263$, we set $\omega=0.005$; for Example 2,
$\omega_{norm}\approx 0.021370788062004$, we set $\omega=0.02$.

The Landweber method (\ref{Landweber}) has the form
\begin{equation}\label{LandweberBLT}
\left\{\begin{array}{l}
f^{k+1} = f^{k} - \Delta t (w^{k}_D-w^{k}_N)\chi_{\Omega_0} ,\\
f^{0}=f_0
\end{array}\right.
\end{equation}
with $0<\Delta t<2\omega_{norm}$. For Example 1, we set $\Delta t=2\times 0.005=0.01$; for Example 2, we set $\Delta t=2\times 0.02=0.04$.

As suggested by Subsection \ref{subsec:parameter}, in our methods (\ref{symplecticBLT}) (termed as ``ARM'') and (\ref{symplectic2BLT}) (termed as ``NSS''),  we set $s=2$, $\Delta t = 0.0625$ for Example 1 and $\Delta t = 0.125$ for Example 2. In all methods, the initial guess $f_0=0$ and the iterations stop when $\|y^{h,\delta}-Af^{h,k}\|_{V} =\|u^h_D(f^{h,k},g^\delta_1)-u^h_N(f^{h,k},g^\delta_2)\|_{V} \leq \tau \delta$
with $\tau=1.1>1$ for Example 1 and $\tau=10>1$ for Example 2. We note that when the noisy level is large, bigger $\tau$ is suggested so
that iterations can stop before the solution accuracy gets worse.

The results of the simulations are presented in Table \ref{tab:comparison}, from which
we conclude that, with properly chosen parameters, all the above mentioned methods are stable and can produce satisfactory solutions.
Moreover, on one hand, compared with the conventional Landweber method, all of the other methods produce better accuracy with considerably fewer iterations; on the other hand, for the given model problems, the proposed ARM and NSS are comparable to well-known accelerated regularization methods, i.e. the $\nu$-method and the Nesterov's method, in both solution accuracy and convergence rate.

\begin{table}[H]
{\footnotesize
\caption{Comparison of different methods. The CPU time is measured in seconds.}\label{tab:comparison}
\begin{center}
\begin{tabular}{|c|c|c|c|c|c|c|} \hline
$\delta'$ &
\multicolumn{2}{c|}{$0.5\%$} &
\multicolumn{2}{c|}{$1\%$} &
\multicolumn{2}{c|}{$5\%$} \\ \hline
\multicolumn{7}{|c|}{\textbf{Example 1}} \\  \hline
\cline{2-7}
Methods   & L2Err$_{k^*}$ & $k^*$ (CPU)  & L2Err$_{k^*}$ & $k^*$ (CPU) & L2Err$_{k^*}$ & $k^*$ (CPU) \\ \hline
Landweber &  4.7418e-3  & 5079(149.61) &   5.5450e-3 & 3735(109.83)   & 5.7679e-3 & 2787(88.44)      \\
$\nu=0.25$     & 2.0463e-1 & 18594(541.84)    & 4.2972e-2 & 2703(80.31) 	   & 4.3583e-2 & 615(19.78)     \\
$\nu=0.5$     & 1.1026e-2 & 1706(52.78)   & 8.0824e-3 &314(9.28)	   & 7.6839e-3 & 186(5.97)    \\
$\nu=1$     & 4.6974e-3 &  218(6.78)  & 4.7379e-3 & 89(2.81) 	   & 5.0483e-3 & 86(2.72)    \\
$\nu=2$      & 4.7277e-3 & 138(4.17)  & 5.2199e-3 &135(4.86) 	   & 5.2410e-3 & 128(4.02)     \\
Nesterov    & 4.7724e-3 & 787(23.08) & 5.0624e-3 &	283(9.02)	   & 1.0487e-3 & 152(4.92)        \\
NSS       & 4.9850e-3  & 236(7.53)  & 6.1620e-3 &	232(6.97)    & 6.5021e-3 & 223(6.75)       \\
ARM       & 4.9968e-3 & 236(7.78)   & 6.1991e-3 & 232(6.91)     & 6.5776e-3 & 223(6.78)   \\ \hline
\multicolumn{7}{|c|}{\textbf{Example 2}}  \\  \hline
Methods   & L2Err$_{k^*}$ & $k^*$ (CPU)      & L2Err$_{k^*}$ & $k^*$ (CPU) & L2Err$_{k^*}$ & $k^*$ (CPU)  \\ \hline
Landweber & 7.7560e-2   & 7534(236.28) &  7.8878e-2 & 6857(217.75) & 1.1761e-1 & 723(22.84)     \\
$\nu=0.25$     & 3.7686e-2 & 955(31.28)   &  3.7667e-2 & 955(30.67)	   &  7.5047e-2 &  132(4.61)     \\
$\nu=0.5$    & 6.3875e-2 & 214(7.31)    & 6.7694e-2 &	179(5.95)        & 1.1228e-1 & 67(2.39)       \\
$\nu=1$      & 7.5760e-2 & 190(6.25)  &  7.5604e-2 & 188(6.23)	   &  1.1739e-1 &  47(1.73)    \\
$\nu=2$     & 7.8135e-2 & 246(8.11)   &  7.9302e-2 & 236(7.95) 	   &  1.1921e-1 &  67(2.36)    \\
Nesterov    & 7.2335e-2 & 365(11.55)  & 7.1938e-2 & 363(11.94)        & 1.1715e-1 & 8(2.75)         \\
NSS        &  7.8257e-1 & 450(16.94)  & 7.9035e-2 &	436(13.97)        & 1.2095e-1 & 121(4.23)	         \\
ARM        & 7.8275e-2 & 450(14.22)  & 7.8956e-2 & 437(14.52)        & 1.212e-1 & 121(4.06)        \\ \hline
\end{tabular}
\end{center}
}
\end{table}

We note that inverse source problems with only one measurement on the boundary have infinite solutions. In the context of the BLT problem, one cannot distinguish between a strong source
over a small region and a weak source over a large region. In this paper, we are interested in the minimum norm solution, which is always unique for our linear inverse problems. However, in practice, we don't know which solution is the minimum norm solution. Therefore, in all the above experiments, we suppose that the phantom, which has been used to generate data, is just the minimum norm solution. In order to make this assumption acceptable, we are assumed to know exactly the positions of sources (i.e. the geometry of $\Omega_0$). With the help of appropriately selected $\Omega_0$ (e.g. the one we used in our examples), we found that in the noise-free case, for large $k$ (e.g. $k=1000$), $f_k$ is very close (in the accuracy L2Err$_k <$ 1e-5) to the input source $f^*$. Since $f_k\to f^\dagger$ as $k\to \infty$, we can assume that $f^\dagger=f^*$.

Of course, one can also apply the proposed method to other linear inverse problems
and compare the behavior of different methods. Another well-known linear inverse problem is the Cauchy problem of
finding $(\phi, t)$ on unaccessible boundary $\Gamma_u$ from the Cauchy data $(\Phi, T)$ on accessible boundary $\Gamma_a$
such that the following relations hold:
\[
\left \{
\begin{array}{cl}
 - \Delta u= 0 & \textmd{in}\ \Omega,\\[0.5em]
\partial_\nu u = \Phi,\quad u=T & \textrm{on}\ \Gamma_a,\\[0.5em]
\partial_\nu u = \phi,\quad u=t & \textrm{on}\ \Gamma_u.
\end{array}
\right.
\]
In contrast to the BLT problem, the Cauchy problem above admits solution uniqueness, provided a solution exists. We can expect good behavior of the proposed method for this Cauchy problem. However, for the conciseness of the paper, we omit these numerical results.

\section{Conclusions}

In this paper we have proposed a new class of accelerated regularization methods for solving ill-posed linear operator equations. A series of theoretical results including limiting behavior and convergence rates are proved. Moreover, as an application of the proposed method, in this paper, a model problem arising from bioluminescence tomography is discussed in detail. Since the proposed methods are comparable to the Nesterov's acceleration method and the $\nu$-method about the convergence rate and the solution accuracy, they are promising approaches which merit further theoretical and numerical development as well as more extensive comparison to state-of-the-art methods. Similar to Nesterov's acceleration method \cite{Nesterov1983,Ramlau2018} or it's modified versions \cite{Hubmer2017,Jin2016}, the introduced iterative regularization methods can also be used to solve to some non-linear ill-posed problems. However, for performing a rigorous theoretical analysis, the concept of acceleration in the sense of regularization theory should be extended so that it can be used for evaluating general non-linear regularization methods.

\bigskip



\section*{Acknowledgement}

We express our gratitude to the anonymous referees whose valuable comments and suggestions allowed us to eliminate weak points of the manuscript and thus to improve the paper.

The work of R. Gong is supported by the Natural Science Foundation of China (No. 11401304, 11971230) and the Fundamental Research Funds for the Central Universities (No. NS2018047). The work of B. Hofmann is supported by the German Research Foundation (DFG-grant HO 1454/12-1), and the work of Y. Zhang is supported by the Alexander von Humboldt foundation through a postdoctoral researcher fellowship.


\section*{Appendix A: Proof of Proposition \ref{SolutionODE}}

For simplicity, we only consider the case of positive $s$. For the case $s\in(-1/2,0)$, we refer to the similar result, presented in \cite[Lemma 6.1]{Botetal18}. The general solution to (\ref{SVDEq2}) is
\begin{eqnarray}\label{GeneralSolution}
\xi_j(t) = \lambda^{-1}_j \langle y^\delta , v_j \rangle + \left\{\begin{array}{ll}
\frac{C^j_{1,s}}{(\lambda_j t)^s} J_s(\lambda_j t) + \frac{C^j_{2,s}}{(\lambda_j t)^s} Y_s(\lambda_j t), \textrm{~if~} s\in \mathbb{N}\cup \{0\},\\
\frac{C^j_{1,s}}{(\lambda_j t)^s} J_s(\lambda_j t) + \frac{C^j_{2,s}}{(\lambda_j t)^s} J_{-s}(\lambda_j t), \textrm{~if~} s\not\in \mathbb{N},
\end{array}\right.
\end{eqnarray}
where $Y_s$ denotes the Bessel functions of second kind. In order to determine the constants $C^j_{1,s}$ and $C^j_{2,s}$ from the initial conditions, we distinguish three different cases: (i) $s=0$, (ii) $s\in \mathbb{N}$, and (iii) $s\not\in \mathbb{N}$. We show that for all of three cases $C^j_{2,s}=0$ according to the boundedness of initial data. In case (i), by using the divergence behaviour $Y_0(\lambda_j t)=\mathcal{O}(\log(\lambda_j t))$ as $t\to0$ \cite[(9.1.12)]{Abramowitz1972}, $C^j_{2,0}$ must be zero. For case (ii), the asymptotic, cf. \cite[(9.1.11)]{Abramowitz1972}, $Y_s(\lambda_j t)=\mathcal{O}((\lambda_j t)^{-s})$ as $t\to0$ implies $\frac{Y_s(\lambda_j t)}{(\lambda_j t)^s}=\mathcal{O}((\lambda_j t)^{-2s})$. Therefore, $C^j_{2,s}=0$ for all $s\in \mathbb{N}$. Now, consider the last case. According to the asymptotic $J_{-s}(\lambda_j t)=\mathcal{O}((\lambda_j t)^{-s})$ as $t\to0$, cf. \cite[(9.1.10)]{Abramowitz1972}, $C^j_{2,s}=0$ for all $s\not\in \mathbb{N}$.

By the above analysis, the general solution to (\ref{SVDEq2}) bounded initial data should be
\begin{eqnarray*}
\xi_j(t) = \frac{C^j_{1,s}}{(\lambda_j t)^s} J_s(\lambda_j t) + \lambda^{-1}_j \langle y^\delta , v_j \rangle.
\end{eqnarray*}
By the initial data $f_0=\sum_j \langle f_0,u_j \rangle u_j$ and the limit $\lim_{t\to0} \frac{J_s(\lambda_j t)}{(\lambda_j t)^s}= \frac{1}{2^s\Gamma(s+1)}$, we conclude that
\begin{eqnarray*}
C^j_{1,s} =  2^s\Gamma(s+1) \left( \langle f_0,u_j \rangle - \lambda^{-1}_j \langle y^\delta , v_j \rangle \right),
\end{eqnarray*}
which gives the desired formula for $\xi_j(t)$.

Finally, check that $\dot{\xi}_j(0)=0$. It can be done by the following limit
\begin{eqnarray*}
\dot{\xi}_j(0) = \lim_{t\to0+} \frac{\left( 1- 2^s\Gamma(s+1) \frac{J_s(\lambda_j t)}{(\lambda_j t)^s} \right) \left(  \lambda^{-1}_j \langle y^\delta , v_j \rangle - \langle f_0,u_j \rangle \right)}{t} =0
\end{eqnarray*}
by noting that \cite[(9.1.10)]{Abramowitz1972}
\begin{eqnarray}\label{Jlimit}
\frac{J_s(\lambda_j t)}{(\lambda_j t)^s}=\frac{1}{2^s\Gamma(s+1)} + \mathcal{O}((\lambda_j t)^{2}) \textrm{~as~} t\to0.
\end{eqnarray}

\section*{Appendix B: Proof of Lemma \ref{Rootdiscrepancy}}

This proof uses the technique in~\cite{Attouch2018}. Consider the Lyapunov function of (\ref{SecondFlow}) by $\mathcal{E}(t) = \frac{1}{2} \|\dot{f}^\delta(t)\|^2+ \|K f^\delta(t) - y^\delta \|^2$. It is easy to show that
\begin{equation}\label{LyapunovDerivative2}
\dot{\mathcal{E}}(t)= - \frac{1+2s}{t} \|\dot{f}^\delta(t)\|^2 \leq 0.
\end{equation}
Hence, $\mathcal{E}(t)$ is non-increasing, and $\mathcal{E}(\infty) := \lim_{t\to \infty} \mathcal{E}(t)$ exists by noting that $\mathcal{E}(t)\geq0$ for all $t$. Now, consider the function $e(t)=\frac{1}{2} \|f^\delta(t) - f^*\|^2$. It is not difficult to obtain
\begin{equation}
\label{ProofIneqNestorovE}
\ddot{e}(t) + \frac{1+2s}{t} \dot{e}(t) + \|K f^\delta(t) - y^\delta \|^2 \leq  \|\dot{f}^\delta(t)\|^2.
\end{equation}
Divide this expression by $t$ to obtain
\begin{eqnarray*}
\frac{1}{t} \ddot{e}(t) + \frac{1+2s}{t^2} \dot{e}(t) + \frac{1}{t} \mathcal{E}(t) \leq \frac{3}{2t} \|\dot{f}^\delta(t)\|^2,
\end{eqnarray*}
Integrating the above inequality from $1$ to $t$ and using integration by parts for $\ddot{e}(t)$, we obtain
\begin{equation}
\label{ProofIneq1}
\int^t_{1} \frac{\mathcal{E}(\tau)}{\tau} d\tau \leq  \dot{e}(1) - \frac{\dot{e}(t)}{t} - 2(1+s) \int^t_{1} \frac{\dot{e}(\tau)}{\tau^2} d\tau + \frac{3}{2} \int^t_{1} \frac{\|\dot{f}^\delta(\tau)\|^2}{\tau}  d\tau .
\end{equation}
On the one hand, using the integration by parts and the positivity of functional $e(\cdot)$, we have
\begin{equation}
\label{ProofIneq2}
\int^t_{1} \frac{\dot{e}(\tau)}{\tau^2} d\tau = \frac{e(t)}{t^2} - e(1) + 2 \int^t_{1}  \frac{e(\tau)}{\tau^3} d\tau \geq - e(1) .
\end{equation}
On the other hand, relation (\ref{LyapunovDerivative2}) gives
\begin{equation}
\label{ProofIneq3}
 \int^t_{1} \frac{\|\dot{f}^\delta(\tau)\|^2}{\tau}  d\tau = \frac{\mathcal{E}(1) - \mathcal{E}(t)}{1+2s}.
\end{equation}

Combine (\ref{ProofIneq1})-(\ref{ProofIneq3}) to get
\begin{eqnarray}
\label{ProofIneq4}
\int^t_{1} \frac{\mathcal{E}(\tau)}{\tau} d\tau \leq \dot{e}(1) - \frac{\dot{e}(t)}{t} + 2(s+1) e(1) + \frac{3(\mathcal{E}(1) - \mathcal{E}(t))}{2(1+2s)} \\ \qquad
= C(1) - \frac{\dot{e}(t)}{t} - \frac{3\mathcal{E}(t)}{2(1+2s)},
\end{eqnarray}
where $C(1)= \dot{e}(1) + 2(s+1) e(1) + \frac{3\mathcal{E}(1)}{2(1+2s)}$ collects the constant terms. Therefore, for any $T\geq t> 1$, we have
\begin{equation}
\label{ProofIneq5}
\mathcal{E}(T) \int^t_{1} \frac{1}{\tau} d\tau + \frac{3\mathcal{E}(T)}{2(1+2s)} \leq C(1) - \frac{\dot{e}(t)}{t}
\end{equation}
by noting the non-increasing of Lyapunov function $\mathcal{E}(t)$. Rewrite (\ref{ProofIneq5}) as $\mathcal{E}(T) \left( \ln(t)+ \frac{3}{2(1+2s)} \right) \leq C(1) - \frac{\dot{e}(t)}{t}$, and then integrate it from $t=1$ to $t=T$ to derive
\begin{eqnarray}
\mathcal{E}(T) \left( T \ln(T) + 1 - T + \frac{3}{2(1+2s)} (T - 1) \right) \nonumber \\ \qquad\qquad \leq C(1)(T - 1) - \int^T_{1} \frac{\dot{e}(t)}{t} dt. \label{ProofIneq6}
\end{eqnarray}

Moreover, using the integration by parts and the positivity of functional $e(\cdot)$, we have
\begin{equation}
\label{ProofIneq7}
\int^T_{1} \frac{\dot{e}(\tau)}{\tau} d\tau = \frac{e(T)}{T} - e(1) + \int^T_{1}  \frac{e(t)}{t^2} dt \geq - e(1).
\end{equation}

By combining (\ref{ProofIneq6}) and (\ref{ProofIneq7}), we deduce that
\begin{equation}
\label{ProofIneq8}
\mathcal{E}(T) \left( T \ln(T) + C_1 T + C_2 \right) \leq  C(1) T + C_3,
\end{equation}
where $C_1=\frac{3}{2(1+2s)}-1$, $C_2=-C_1$ and $C_3=e(1)- C(1)$ are three constants.

Inequality (\ref{ProofIneq8}) immediately yields $\mathcal{E}(\infty)\leq0$. By the non-negativity of Lyapunov function $\mathcal{E}(\cdot)$, we conclude
\begin{equation}
\label{LimitEnergy}
\mathcal{E}(\infty)=0.
\end{equation}

The continuity of $\chi(T)$ is obvious as our problem is linear. Hence, from (\ref{LimitEnergy}) and the assumption of the lemma, we conclude that
\begin{eqnarray*}\label{TwoLimits}
\lim_{T\to \infty} \chi(T) \leq (1-\tau)\delta  <0 \quad \textrm{~and~} \quad \chi(0) = \|K f_0-y^\delta\| - \tau \delta >0,
\end{eqnarray*}
which implies the existence of the root of $\chi(T)$.

\section*{Appendix C: Proof of Proposition \ref{ThmCompact}}

Let $\{f^n\}_{n}\subset Q_0$ be bounded. Then there is a subsequence,
denoted again by $\{f^n\}_{n}$, which converges weakly in $Q_0$ to
some element $f^*\in Q_0$ because of the reflexivity of space $Q_0$. Let $u^n_D=u_D(f^n)$, $u^n_N=u_N(f^n)$, i.e., $u_D^n\in V_0,
u_N^n\in V$, and
\begin{eqnarray}
a(u^n_D,v) = \langle f^n,v\rangle_{Q_0}\quad\forall\,v\in V_0,\label{2.13}\\
a(u^n_N,v) = \langle f^n,v\rangle_{Q_0}\quad\forall\,v\in V.\label{2.14}
\end{eqnarray}
Then $\{u_D^n\}_n$ and $\{u_N^n\}_n$ are bounded in $V$ from the
properties (\ref{prioriestimated1}) and (\ref{prioriestimaten1}).
Hence, we can extract two further subsequences, denoted again by
$\{u_D^n\}_n$ and $\{u_N^n\}_n$, which converge weakly in $V$ and
strongly in $Q$ to $u_D^*\in V_0$ and $u_N^*\in V$, respectively. Let
$n\to \infty$ in (\ref{2.13}) and (\ref{2.14}) to get
$u_D^*=u_D(f^*)$ and $u_N^*=u_N(f^*)$. Strong convergence of
$\{u_D^n\}_n$ to $u_D^*$ in $V$ follows from
\begin{eqnarray*}
\|u^n_D-u^*_D\|^2_V = a(u^n_D-u^*_D,u^n_D-u^*_D)
=\int_{\Omega_0}(f^n-f^*)\,(u^n_D-u^*_D)\,dx \to 0
\end{eqnarray*}
as $n\to \infty$.  Similarly, $u_N^n\to u^*_N$ as $n\to\infty$.

Denote $g^n=K\,f^n$. Then $\{g^n\}_{n}$ is bounded in $V$.
Repeating the above argument, we conclude that there exists an element $s^*\in V$
such that
\[  g^n \rightharpoonup g^*\ \mbox{in $V$} \quad \mbox{as} \quad n\rightarrow \infty.  \]
Therefore, $\forall\ v\in V$,
\begin{eqnarray*}
&\langle Kf^*-g^*,v\rangle_V =\lim_{n\to \infty} \langle Kf^*-g^n,v\rangle_V =\lim_{n\to \infty}  \langle u_D(f^*)-u_N(f^*)-g^n,v\rangle_{V} \\
& \qquad =\lim_{n\to \infty} \langle u_D(f^n)-u_N(f^n)-g^n,v\rangle_{V}=\lim_{n\to \infty} \langle K f^n-g^n,v\rangle_{V}=0.
\end{eqnarray*}
Thus, we have $g^*=K f^*$.
Consequently, strong convergence of $g^n$ to $g^*$ in $V$ follows from
\begin{eqnarray*}
\|g^n-g^*\|^2_V & =
\|K f^n-K f^*\|^2_V \\
&=\|u_D(f^n)-u_D(f^*)-(u_N(f^n)-u_N(f^*))\|_V^2\\
&\leq 2\|u_D(f^n)-u_D(f^*)\|_V^2+2\|u_N(f^n)-u_N(f^*)\|_V^2 \rightarrow 0
\end{eqnarray*}
as $n\rightarrow \infty$, and the proof is completed.

\section*{Appendix D: Finite element discretization of boundary value problems}

In this appendix, we discuss the numerical implementations of (\ref{FinitePro_uD}) and (\ref{FinitePro_uN}) by standard
finite element method.  We use linear finite element space for an approximation of the
light source space $Q_0$. Specifically, let $\{\mathcal {T}_{0,H}\}_H$
be a regular family of triangulations over domains
$\overline{\Omega}_0\subset \overline{\Omega}$ with meshsize $H>0$.
For each triangulation $\mathcal{T}_{0,H}=\{K_H\}$, define finite
element space $Q^H_0=\{q\in C(\overline{\Omega}_0)\mid q|_{K_H}\in \mathcal{P}_{1}(K),\
\forall\,K_H\in\mathcal{T}_{0,H}\}$, where $\mathcal{P}_k$ represents the space of all polynomials of degree no greater than $k$.
Let $\{\mathcal{T}_h\}_{h}$ be a regular family of triangulations
over domains $\overline{\Omega}\subset \mathbb{R}^d$ with a
mesh size $h>0$. For each triangulation $\mathcal{T}_h=\{K_h\}$,
define finite element spaces $V^h$ and $V_0^h$ as follows.
\begin{eqnarray*}
V^h := \{v\in C(\overline{\Omega})\mid v\mid_{K_h} \in
\mathcal{P}_1,\ \forall\,K_h\in \mathcal{T}_h\},\quad  V^h_0 = V^h\cap
V_0.
\end{eqnarray*}
Moreover, we use the symbol $g^\delta_1+V_0^h$ for the set
\[ \{v \in V^h \mid v(x_i)=g^\delta_1(x_i)\ \forall \,
\textmd{vertex}\ x_i \in K_h\cap\Gamma,\,\forall\,K_h
\in\mathcal{T}_h\}.
\]

For each $f\in Q_0$, the finite element discretization of (\ref{FinitePro_uD})
and (\ref{FinitePro_uN}) read
\begin{eqnarray}
u^h_D:=u^h_D(f,g^\delta_1)\in g^\delta_1+V_0^h,\quad a(u^h_D,v) =\langle f,v\rangle_{Q_0}\quad\forall\,v\in V^h_0,\label{apprDirichlet}\\
u^h_N:=u^h_N(f,g^\delta_2)\in V^h,\quad a(u^h_N,v) = \langle f,v\rangle_{Q_0}+\langle g^\delta_2,v\rangle_{Q_\Gamma}\quad\forall\,v\in V^h.\label{apprNeumann}
\end{eqnarray}
Similar to the continuous case, we use the symbols $u^{h}_D(f)$,
$\widetilde{u}^{h}_D(g^\delta_1)$, $u^{h}_N(f)$ and
$\widetilde{u}^{h}_N(g^\delta_2)$ for $u^{h}_D(f,0)$,
$\widetilde{u}^{h}_D(0,g^\delta_1)$, $u^{h}_N(f,0)$ and $u^{h}_N(0,g^\delta_2)$,
respectively.

Suppose that $\mathcal{T}_{0,H}$ and $\mathcal{T}_h$ are consistent, i.e., the triangulation
$\mathcal{T}_{0,H}$ is a restriction of the triangulation $\mathcal{T}_h$ on
$\overline{\Omega}_0$, and let $n_0$ and $n$ be the numbers of nodes of the triangulations
$\mathcal{T}_{0,H}$ and $\mathcal{T}_h$. Denote $\varphi_i(x)\in V^h$, $1\le i\le n$, be the node basis
functions of the finite element space $V^h$ associated with grid nodes
$x_i\in \overline{\Omega}$. Let $x_{i_j}\in \overline{\Omega}_0, 1\leq j\le n_0$ be the nodes of $\mathcal{T}_{0,H}$,
and $\varphi_{i_j}(x)\in V^h$ the corresponding basis functions. Then, the approximate source function $f^H$
of $f$ can be expressed by $f^H=\sum^n_{j=1}f_j^H \varphi_{i_j}$ with
$f_j^H=f(x_{i_j})$. For the problems (\ref{apprDirichlet}) and
(\ref{apprNeumann}), the solutions $u_D^h\in g^\delta_1+V_0^h$ and
$u_N^h\in V^h$ can be expanded by $u_D^h=\sum^n_{i=1}u_{D,i}
\varphi_i$ and $u_N^h=\sum^n_{i=1}u_{N,i} \varphi_i$, respectively,
where $u_{D,i}=u_D^h(x_i)$ and $u_{N,i}=u_N^h(x_i)$.

Denote $I=\{1,2,\cdot\cdot\cdot,n\}$, $I_0=\{1,2,\cdot\cdot\cdot,n_0\}$, $I_b=\{i\in I| x_i\in\Gamma\}$,
and define
\begin{eqnarray*}
& S=(s_{ji}),\quad s_{ji}=\int_\Omega D\,\nabla \varphi_i
\nabla \varphi_j\,dx, \ i,j\in I,\\
& M=(m_{ji}),\quad m_{ji}=\int_\Omega
\mu_a\,\varphi_i\,\varphi_j\,dx,\ i,j\in I,\\
& M_0=(m^0_{jk}),\quad m^0_{jk}=\int_\Omega
\mu_a\,\varphi_{i_k}\,\varphi_j\,dx,\ j\in I, k\in I_0,\\
& z = (z_1,z_2,\cdot\cdot\cdot,z_n)^t, \quad
z_j=\int_{\Gamma}g^\delta_2\,\varphi_j\,ds, \qquad L=S+M.
\end{eqnarray*}
In the following, we use the same symbol for a finite element
function and its vector representation associated with the given
finite element basis functions. Then, the finite element solutions $u_D^k$ and $u_N^k$ of the forward problems (\ref{apprDirichlet}) and (\ref{apprNeumann}) corresponding to the source $f^k$, can be calculated by
\begin{eqnarray*}
& L\,u^k_D=M_0\,f^k,\quad u^k_{D,i} =g^\delta_1(x_i), i\in I_b,\quad
u_D^k=\sum^n_{i=1}u^k_{D,i} \varphi_i,\\
& L\,u^k_N = M_0\,f^k+z,\quad u_N^k=\sum^n_{i=1}u^k_{N,i} \varphi_i.
\end{eqnarray*}

Similarly, for the discretization of the quality $K^*\,(K\,f-y^\delta)$, define
\[ C = (c_{ji}),\quad c_{ji} = \int_\Omega \varphi_i \varphi_j\,dx,\ i,j\in I. \]
Then the finite element approximation of $K^*\,(K\,f^k-y^\delta)=w^k_D-w^k_N$ can be calculated through
\begin{eqnarray}
& L\,w^k_D=C\,(u^k_D-u^k_N),\quad\ w^k_{D,i} =0,\ i\in I_b,
\quad w_D^k=\sum^n_{i=1}w^k_{D,i} \varphi_i, \label{FiniteEq_wD} \\
 & L\,w^k_N=C\,(u^k_D-u^k_N),\quad w_N^k=\sum^n_{i=1}w^k_{N,i}\varphi_i. \label{FiniteEq_wN}
\end{eqnarray}
Note that (\ref{FiniteEq_wD}) and (\ref{FiniteEq_wN}) are the finite element discretization of the
adjoint problems (\ref{adjointd}) and (\ref{adjointn}) with $u_{DN}$ being replaced by $u^k_D-u^k_N$.


\section*{References}
\bibliographystyle{plain}
\bibliography{ZhaGon19}
\end{document}